\documentclass[a4paper,12pt]{amsart}
\pdfoutput=1

\usepackage{amssymb,amsfonts,amsmath,amsthm}
\usepackage{amsrefs,mathtools,stmaryrd}
\usepackage{xcolor}
\usepackage{float}
\usepackage{tikz}
\usepackage{verbatim}
\usepackage[english]{babel}
\usepackage[all]{xy}
\usepackage{doi}
\usepackage{booktabs}
\usepackage{etoolbox}

\usepackage[margin=1in]{geometry}

\def\={\; = \;}
\def\+{\; + \;}

\def\:{\; \colon \;}

\newcommand{\Mod}[1]{\ (\text{mod}\ #1)}

\newcommand{\es}[1]{\; #1 \;}

\newcommand{\sm}{\smallsetminus}

\newcommand{\RR}{{\mathbb{R}}}
\newcommand{\ZZ}{{\mathbb{Z}}}
\newcommand{\PP}{{\mathbb{P}}}
\newcommand{\CC}{{\mathbb{C}}}
\newcommand{\FF}{{\mathbb{F}}}
\newcommand{\TT}{{\mathbb{T}}}

\newcommand{\QQ}{{\mathbb{Q}}}

\newcommand{\tr}{\operatorname{tr}}
\newcommand{\diag}{\operatorname{diag}}

\newcommand{\SL}{\operatorname{SL}}
\newcommand{\GL}{\operatorname{GL}}
\newcommand{\PSL}{\operatorname{PSL}}
\newcommand{\Sp}{\operatorname{Sp}}

\newcommand{\Gal}{\operatorname{Gal}}
\newcommand{\disc}{\operatorname{disc}}
\newcommand{\irr}{{\operatorname{irr}}}

\newcommand{\calM}{{\mathcal M}}

\newcommand{\calG}{{\mathcal G}}
\newcommand{\calO}{{\mathcal O}}

\theoremstyle{plain}

\newtheorem{proposition}{Proposition}

\title{Goursat rigid local systems of rank four}
\date{}
\author{Danylo~Radchenko}
\address{Max Planck Institute for Mathematics\\
Vivatsgasse 7\\
53111 Bonn, Germany}
\email{danradchenko@gmail.com}
\author{Fernando~Rodriguez~Villegas}
\address{The Abdus Salam International Centre for Theoretical Physics\\
Strada Costiera 11\\
Trieste 34151, Italy}
\email{villegas@ictp.it}

\begin{document}
\maketitle
\begin{center}
{\it Dedicated to Don Zagier for his $65$th birthday}
\end{center}

\begin{abstract}
  We study the general properties of certain rank $4$ rigid local
  systems considered by Goursat. We analyze when they are irreducible,
  give an explicit integral description as well as the invariant
  Hermitian form $H$ when it exists. By a computer search we find what
  we expect are all irreducible such systems all whose solutions are
  algebraic functions and give several explicit examples defined over
  $\QQ$.  We also exhibit one example with infinite monodromy as
  arising from a family of genus two curves.
\end{abstract}
\section{Introduction}

The question of when linear differential equations in a variable $t$
have all of their solutions algebraic functions of $t$
goes back to the early 1800's. In his 1897 thesis written under the
supervision of
P.~Painlev\'e\cite{Boulanger},\cite{Painleve},\cite{Singer},
A. Boulanger mentions a paper by~J. Liouville of 1833~\cite{Liouville}
as a possible first work on the matter. The introduction of
Boulanger's thesis offers a lucid description of the history of the
question up to the time of his writing.

Schwarz~\cite{Schwarz} famously described all cases of algebraic
solutions to the hypergeometric equation satisfied by Gauss's series
${}_2F_1$. This was much later extended to hypergeometric equations of
all orders by Beukers and Heckman~\cite{Beukers-Heckman}. In what
follows we will often refer to the better known hypergeometric local
systems for comparison.

From a broader point of view, we may say that differential equations
with all solutions algebraic are a special case of motivic local
systems.  Simpson conjectures in~\cite[p. 9]{Simpson-1} that all rigid
local systems satisfying some natural conditions are motivic. Without
attempting a rigorous definition of what this means, we will just say
that such systems should be geometric in nature. This is known for
rigid local systems on $\PP^1$
by the work of Katz~\cite{Katz}, who gave a general algorithm (using
middle convolution) for their construction. See
also~\cite{Esnault-Groechenig} for systems over a higher dimensional
base and~\cite{Zagier-1} for more on differential equations and
arithmetic.

In this note we consider the case of Goursat's case II of rank $4$
rigid local systems (denoted henceforth by G-II). These correspond to
order $4$ linear differential equations with three regular singular
points, say $t=0,1,\infty$, with semisimple, finite order local
monodromies with eigenvalues of multiplicities $2 1^2,2^2,1^4$
respectively.

We study the general properties of G-II systems; for example, we
analyze when they are irreducible and describe the invariant Hermitian
form $H$ when it exists. As in~\cite{Beukers-Heckman} $H$ is a key
tool to understand when the monodromy group is finite. Indeed, a
necessary condition is that $H$ be definite in every complex embedding
of the field of definition. Finiteness of the monodromy group is
equivalent to solutions to the linear differential equations being
algebraic.

We also show explicitly that the monodromy group can be defined in an
integral way in terms of the eigenvalues of the local monodromies (the
defining data). This would follow if the rigid system was motivic
(see~\cite[p.~9]{Simpson-1}). We find (see~\S\ref{cocycle}) that there
is a non-trivial obstruction for the field of definition of the
monodromy group. It may not be possible to define the monodromy group
in its field of moduli (the field of coefficients of the
characteristic polynomials of the local monodromies). This is in
contrast with the hypergeometric case, for example, where by a theorem
of Levelt~\cite[Prop.~3.3]{Beukers-Heckman} such an obstruction does
not occur. For the G-II systems the obstruction is given by a
quaternion algebra over the field of moduli.

By a computer search we find what we expect are all irreducible G-II
equations whose solutions are algebraic functions and give several
explicit examples defined over $\QQ$.  We also exhibit one example
with infinite monodromy as arising from a family of genus two curves.

We present in this paper our results with few detailed proofs, which
will appear in a subsequent work.  We used MAGMA~\cite{MAGMA} and
PARI-GP~\cite{PARI2} for most of the calculations.

\section*{Acknowledgements}
This work was started at the Abdus Salam Centre for Theoretical
Physics and completed during the special trimester {\it Periods in
  Number Theory, Algebraic Geometry and Physics} at the Hausdorff
Institute of Mathematics in Bonn, Germany. We would like to thank
these institutions as well as the Max Planck Institute for Mathematics
in Bonn for their hospitality and support. 

The second author would like to thank N.~Katz and D.~Roberts for
useful exchanges regarding the subject of this work.

\section{Rigid local systems}
\label{rls}

Following the setup and notation of~\cite{HLRV} we consider the
character variety $\calM_\mu$ where $\mu$ is an ordered $k$-tuple
of partitions of a positive integer $n$.
This variety parametrizes representations of
$\pi_1(\Sigma\setminus S,*)$
to $\GL_n(\CC)$ mapping a small oriented loop around $s\in S$
to a semisimple conjugacy class $C_s$ 
whose generic eigenvalues have multiplicities
$\mu^s=(\mu^s_1,\mu^s_2,\ldots)$, a corresponding partition in $\mu$.

Here $\Sigma$ is a Riemann surface of genus $g$
and $S$ is a finite set of $k$
points. The eigenvalues are assumed generic in the sense
of~\cite{HLRV}. If non-empty the variety $\calM_\mu$ is equidimensional
of dimension
$$
d_\mu:= (2g-2+k)n^2-\sum_{s\in S}\sum_{i\geq 1} (\mu^s_i)^2 +2.
$$
In this paper we will only consider the case where $g=0$ and in detail
when $k=3,n=4$ and, taking $S=\{0,1,\infty\}$, the partitions are
$\mu^0=2 1^2,\mu^1=2^2,\mu^\infty=1^4$. 

To be concrete, if $g=0$, given conjugacy classes
$C_1,\ldots,C_k\subseteq \GL_n(\CC)$ and labeling the punctures with
$1,\ldots,k$, we are looking for solutions to
\begin{equation}
\label{prod-mat}
T_1\cdots T_k =I_n, \qquad T_s\in C_s, \qquad s=1,\ldots, k,
\end{equation}
where $I_n$ is the identity matrix, up to simultaneous conjugation.
Given such a representation $\pi_1(\Sigma\setminus S,*)$ we call the
image in $\Gamma:=\langle T_1,\ldots, T_k\rangle \leq \GL_n(\CC)$ the {\it
geometric monodromy} group. It is well defined up to
conjugation.

Goursat in his remarkable 1886 paper~\cite{Goursat} discusses when the
local monodromy data uniquely determines the representation, or in
terms of the differential equation and in later terminology, when are
there no {\it accessory parameters}. We want local conditions that
guarantee the following. Given two $k$-tuples of matrices $T_s\in C_s$
and $T'_s\in C_s$ for $s\in S$ satisfying~\eqref{prod-mat} there
exists a single $U\in \GL_n(\CC)$ such that $T_s'=UT_sU^{-1}$ for all
$s\in S$. The corresponding local systems (determined by the local
solutions to the linear differential equation) are known as {\it rigid
  local systems}~\cite{Katz}.

To have a rigid local system is to say that $\calM_\mu$ consists of a
single point. Therefore it is necessary that the expected dimension
$d_\mu$ be zero. This is precisely Goursat's condition~\cite[(5)
p.113]{Goursat} (he only considers the case of $g=0$) as well as
Katz's~\cite{Katz}, which follows from cohomological considerations.

We assume from now on that $g=0$ and then to avoid trivial cases
that $k\geq 3$.  Indeed, for $g=0,k=1$ the group
$\pi_1(\Sigma\setminus S,*)$ is trivial and for $g=0,k=2$ it is
isomorphic to $\ZZ$.  Note, as Goursat points out, that adding an
extra puncture to $S$ with associated partition $(n)$ does not change
the value of $d_\mu$.  Such points correspond to apparent
singularities in the differential equation and may hence be safely
ignored. We will assume then that the partitions $\mu_s$ have at least
two parts.

Goursat shows that with the given assumptions $k\leq n+1$~\cite[top
p.114]{Goursat} and hence there are only finitely many solutions of
$d_\mu=0$ for fixed $n$. He  lists~\cite[p.115]{Goursat} the cases
of $d_\mu=0$ for $n=3$ and  $n=4$ (see below).

It turns out, however, that the condition $d_\mu=0$
is not sufficient as the variety $\calM_\mu$
might be empty. Crawley-Boevey~\cite{Crawley-Boevey} proved that a
necessary and sufficient condition for $\calM_\mu$
to be a point, in the case of generic eigenvalues we are considering,
is that $\mu$
corresponds to a real root of the associated Kac-Moody algebra.

Without getting too deeply into the details of this condition we
present an algorithm that will allow us to determine when $\mu$
represents a real root. This algorithm ultimately corresponds to
Katz's middle convolution.

 It is more convenient to present
the multiplicity data $\mu$
in the form of a star graph with one central node and $k$
legs. We illustrate this in our basic case G-II (Goursat's label II
for $n=4$).

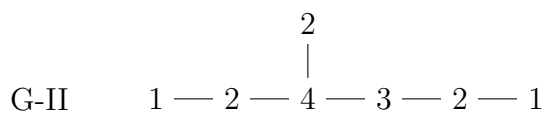
\begin{center}
\begin{tikzpicture}[every node/.style={fill=white}] 
    \draw (-1,0) node[anchor=east]  {G-II};

        \draw (2,1) -- (2,0);

	\draw (0,0) -- (1,0) -- (2,0) -- (3,0) -- (4,0) -- (5,0);

	\node at (2,1) {$2$};
	\node at (0,0) {$1$};
	\node at (1,0) {$2$};
	\node at (2,0) {$4$};
	\node at (3,0) {$3$};
	\node at (4,0) {$2$};
	\node at (5,0) {$1$};

\end{tikzpicture}
\end{center}
The partitions can be read by starting at the central node and moving
away along a leg. The succesive differences of the respective node
values are the parts of the corresponding partition. Nodes with a zero
value are not included; hence each leg has a finite length equal to
that of the corresponding partition.

 The algorithm proceeds starting from a configuration as
 above corresponding to an ordered $k$-tuple $\mu$ of partitions of
 $n$ satisfying $d_\mu=0$ using the following moves.
\begin{itemize}
\item
A: Replace the value $n$ at the central node by 
$$
\sum_i n_i -n,
$$
where $n_i$ are the values at the nodes closest to the central node.
\item
B: Shrink to a point any segment whose endpoints values are the same.
\item C: For each leg put new values on the nodes (not including the
  central node) so that the set of differences of consecutive values
  remains the same but appear in non-decreasing order as one moves
  away from the central node along the leg (so that they correspond to
  a partition of the value at the central node).
\end{itemize}
The goal is to use a sequence of these moves to reach the 
terminal configuration of just a central node with value $1$.
Under the assumptions $d_\mu=0,g=0$ applying A 
strictly decreases the value at the central node and hence the
algorithm always terminates. Indeed, for any partition $\mu=\mu_1\geq
\mu_2\geq\cdots$ of $n$ we have $n\mu_1\geq \sum_i\mu_i^2$. It
follows that if $d_\mu=0$ 
$$
n\sum_s\mu_1^s>(2g-2+k)n^2
$$
and since also $g=0$ and $n>0$ that $\sum_s\mu_1^s>(k-2)n$ which proves the claim.

For our running example $\mu=(2 1^2,2^2,1^4)$ the algorithm works as
follows.

\bigskip
 Apply A:
\begin{center}
\begin{tikzpicture}[every node/.style={fill=white}] 

        \draw (2,1) -- (2,0);

	\draw (0,0) -- (1,0) -- (2,0) -- (3,0) -- (4,0) -- (5,0);

	\node at (2,1) {$2$};
	\node at (0,0) {$1$};
	\node at (1,0) {$2$};
	\node at (2,0) {$3$};
	\node at (3,0) {$3$};
	\node at (4,0) {$2$};
	\node at (5,0) {$1$};

\end{tikzpicture}
\end{center}

Apply B:
\begin{center}
\begin{tikzpicture}[every node/.style={fill=white}] 

        \draw (2,1) -- (2,0);

	\draw (0,0) -- (1,0) -- (2,0) -- (3,0) -- (4,0);

	\node at (2,1) {$2$};
	\node at (0,0) {$1$};
	\node at (1,0) {$2$};
	\node at (2,0) {$3$};
	\node at (3,0) {$2$};
	\node at (4,0) {$1$};

\end{tikzpicture}
\end{center}

Apply C:
\begin{center}
\begin{tikzpicture}[every node/.style={fill=white}] 

        \draw (2,1) -- (2,0);

	\draw (0,0) -- (1,0) -- (2,0) -- (3,0) -- (4,0);

	\node at (2,1) {$1$};
	\node at (0,0) {$1$};
	\node at (1,0) {$2$};
	\node at (2,0) {$3$};
	\node at (3,0) {$2$};
	\node at (4,0) {$1$};

\end{tikzpicture}
\end{center}
We have arrived at the case $\mu=(1^3,21,1^3)$ that corresponds
to the hypergeometric equation of order $3$. It is easy to see that a
next stage takes us to $\mu=(1^2,1^2,1^2)$ corresponding to the
hypergeometric equation of rank $2$ and finally to the desired
terminal case. This confirms that indeed G-II corresponds to a rigid
local system. 

The algorithm fails if at any stage we cannot perform C; i.e. applying
A yields a graph with a central value strictly smaller than one of its
neighbors. As Goursat points out this happens for his case IV.
\begin{center}
\begin{tikzpicture}[every node/.style={fill=white}] 
    \draw (-1,0) node[anchor=east]  {G-IV};

        \draw (1,1) -- (1,0);

        \draw (1,-1) -- (1,0);

	\draw (0,0) -- (1,0) -- (2,0) -- (3,0) -- (4,0);

	\node at (1,1) {$1$};
	\node at (0,0) {$1$};
	\node at (1,-1) {$1$};
	\node at (1,0) {$4$};
	\node at (2,0) {$3$};
	\node at (3,0) {$2$};
	\node at (4,0) {$1$};

\end{tikzpicture}
\end{center}

Apply A:
\begin{center}
\begin{tikzpicture}[every node/.style={fill=white}] 
    \draw (-1,0) node[anchor=east]  {G-IV};

        \draw (1,1) -- (1,0);

        \draw (1,-1) -- (1,0);

	\draw (0,0) -- (1,0) -- (2,0) -- (3,0) -- (4,0);

	\node at (1,1) {$1$};
	\node at (0,0) {$1$};
	\node at (1,-1) {$1$};
	\node at (1,0) {$2$};
	\node at (2,0) {$3$};
	\node at (3,0) {$2$};
	\node at (4,0) {$1$};

\end{tikzpicture}
\end{center}
Since $2<3$ we cannot apply C on the leg going off to the right (one
of the parts would have to be $2-3=-1$).

Here are the diagrams of all rank $n=4$ rigid local systems of the
type in question and their corresponding label in Goursat's
paper (all but the case IV just discussed actually correspond to a
rigid local system). 

\bigskip
\begin{tikzpicture}[every node/.style={fill=white}] 
    \draw (-1,0) node[anchor=east]  {G-I};

        \draw (3,1) -- (3,0);

	\draw (0,0) -- (1,0) -- (2,0) -- (3,0) -- (4,0) -- (5,0) -- (6,0);

	\node at (3,1) {$1$};
	\node at (0,0) {$1$};
	\node at (1,0) {$2$};
	\node at (2,0) {$3$};
	\node at (3,0) {$4$};
	\node at (4,0) {$3$};
	\node at (5,0) {$2$};
	\node at (6,0) {$1$};

\end{tikzpicture}

\bigskip

\begin{tikzpicture}[every node/.style={fill=white}] 
    \draw (-1,0) node[anchor=east]  {G-II};

        \draw (2,1) -- (2,0);

	\draw (0,0) -- (1,0) -- (2,0) -- (3,0) -- (4,0) -- (5,0);

	\node at (2,1) {$2$};
	\node at (0,0) {$1$};
	\node at (1,0) {$2$};
	\node at (2,0) {$4$};
	\node at (3,0) {$3$};
	\node at (4,0) {$2$};
	\node at (5,0) {$1$};

\end{tikzpicture}

\bigskip

\begin{tikzpicture}[every node/.style={fill=white}] 
    \draw (-1,0) node[anchor=east]  {G-III};

        \draw (2,2) -- (2,1) -- (2,0);

	\draw (0,0) -- (1,0) -- (2,0) -- (3,0) -- (4,0);

	\node at (2,2) {$1$};
	\node at (2,1) {$2$};

	\node at (0,0) {$1$};
	\node at (1,0) {$2$};
	\node at (2,0) {$4$};
	\node at (3,0) {$2$};
	\node at (4,0) {$1$};

\end{tikzpicture}

\bigskip
\begin{tikzpicture}[every node/.style={fill=white}] 
    \draw (-1,0) node[anchor=east]  {G-V};

        \draw (1,1) -- (1,0);

        \draw (1,-1) -- (1,0);

	\draw (0,0) -- (1,0) -- (2,0) -- (3,0);

	\node at (1,1) {$2$};
	\node at (0,0) {$1$};
	\node at (1,-1) {$1$};
	\node at (1,0) {$4$};
	\node at (2,0) {$2$};
	\node at (3,0) {$1$};

\end{tikzpicture}

\bigskip
\begin{tikzpicture}[every node/.style={fill=white}] 
    \draw (-1,0) node[anchor=east]  {G-VI};

        \draw (1,1) -- (1,0);

        \draw (1,-1) -- (1,0);

	\draw (0,0) -- (1,0) -- (2,0);

	\node at (1,1) {$2$};
	\node at (0,0) {$2$};
	\node at (1,-1) {$2$};
	\node at (1,0) {$4$};
	\node at (2,0) {$1$};

\end{tikzpicture}

\bigskip
\begin{tikzpicture}[every node/.style={fill=white}] 
    \draw (-1,0) node[anchor=east]  {G-VII};

        \draw (2,0) -- (1,0);

        \draw (1+0.309,0.951) -- (1,0);

	\draw (1-0.809,0.587) -- (1,0);

	\draw (1-0.809,-0.587) -- (1,0);

	\draw (1+0.309,-0.951) -- (1,0);

        \node at (2,0) {$1$};
        \node at (1+0.309,0.951) {$1$};
	\node at (1-0.809,0.587) {$1$};
	\node at (1-0.809,-0.587) {$1$};
	\node at (1+0.309,-0.951) {$1$};

	\node at (1,0) {$4$};

\end{tikzpicture}

\section{Field of definition and field of moduli}
\label{cocycle}

Given a rigid local system with conjugacy classes $C_s$ for $s\in S$
as in~\S\ref{rls} let $q_s(T)$ be the characteristic polynomial of any
element of $C_s$. Let $K$ be the field obtained by adjoining to $\QQ$
the coefficients of all $q_s$. We call $K$ the {\it field of moduli}
or simply the {\it trace field} of the local system (see below for a
justification for this name). It is the smallest field $F$ over which
local monodromies $T_s\in \GL_n(F)$ of the required kind, i.e.,
$T_s\in C_s$, may exist.  But as is typical in such problems it does
not mean that we can actually choose $F=K$.

Given a collection of local monodromies giving rise to our local
system we call its {\it field of definition} the smallest extension
$F$ of $\QQ$ containing all of their entries. We necessarily have
$K\subseteq F$. Note that by Levelt's
theorem~\cite[Prop. 3.3]{Beukers-Heckman}, in the hypergeometric case
we can always take $F=K$, but this is not the case for Goursat's case
II that we analyze here (see~\S\ref{integrality}).

Let $T_s\in C_s$ be a $k$-tuple of matrices in $\GL_n(\bar \QQ)$
satisfying~\eqref{prod-mat}. 
It is clear that for $\sigma\in \Gal(\bar \QQ/K)$ the $k$-tuple
$T_s^\sigma$ is another solution to~\eqref{prod-mat}. Hence by
rigidity there exists $X_\sigma\in \GL_n(\bar \QQ)$ such that
\begin{equation}
\label{conj-local-monodr}
X_\sigma^{-1}T_sX_\sigma=T_s^\sigma, \qquad s\in S.
\end{equation}
Again by rigidity we find that there exists $a_{\sigma,\tau}\in \bar
\QQ$ such that
$$
X_\sigma X_\tau^\sigma=a_{\sigma,\tau}X_{\sigma\tau}.
$$
The map $(\sigma,\tau)\mapsto a_{\sigma,\tau}$ is a $2$-cocycle giving
a well defined element~$\xi \in H^2(\Gal(\bar K/K),K^\times)$.

The following is a standard result.
\begin{proposition}
There exists a solution to~\eqref{prod-mat} over $K$ if and only if
$\xi$ is trivial.
\end{proposition}

Note that~\eqref{conj-local-monodr} implies that the trace of any
product of $T_s$'s is in the trace field $K$. That is, $K$ is indeed
the smallest extension of $\QQ$  containing the traces of all
$T\in\Gamma$.

\section{Explicit solution for the Goursat case II}

In~\cite{Goursat} Goursat writes down an explicit solution 
to~\eqref{prod-mat} for $S=\{0,1,\infty\}$ in the case when $T_0$, $T_1$, and
$T_\infty$ are diagonalizable with spectra $1^2a_1a_2$, $1^2b^2$ and
$c_1c_2c_3c_4$ respectively (assuming that eigenvalues with different
labels are distinct and that $a_1a_2b^2c_1c_2c_3c_4=1$).

Since the triple $(T_0,T_1,T_{\infty})$ is irreducible, the
$1$-eigenspaces for $T_0$ and $T_1$ must have a zero intersection. 
Goursat then shows that in a suitable basis the matrices $T_0$ and $T_1$
are given by
    \begin{equation}
    \label{eq:goursat_form}
    T_0 = \begin{pmatrix}
    1 & 0 & A(1-a_1) & B(1-a_2) \\
    0 & 1 & C(1-a_1) & D(1-a_2) \\
    0 & 0 & a_1 & 0 \\
    0 & 0 & 0 & a_2
    \end{pmatrix},
    \quad\quad
    T_1 = \begin{pmatrix}
    b & 0 & 0 & 0 \\
    0 & b & 0 & 0 \\
    1-b & 0 & 1 & 0 \\
    0 & 1-b & 0 & 1
    \end{pmatrix}.
    \end{equation}
Let $q_{\infty}(x)=(x-c_1)\dots(x-c_4)$ be the characteristic polynomial 
of $T_{\infty}$. A direct computation shows that for given $a_1$, $a_2$, and $b$, the coefficients of $q_{\infty}$ depend linearly on $A$, $D$, 
and $AD-BC$. Conversely, the numbers $A$, $D$, and $AD-BC$ can be found 
from~$q_{\infty}$ by
    \begin{align}
    \begin{split}
    \label{qinf_abcd}
	&\frac{(b-1)(a_1-1)(a_2-a_1)}{b^2a_1^2a_2}\,A \= \frac{a_1^2q_{\infty}(a_1^{-1}) - b^2q_{\infty}(b^{-1})}{a_1-b},\\
    &\frac{(b-1)(a_2-1)(a_1-a_2)}{b^2a_1a_2^2}\,D \= \frac{a_2^2q_{\infty}(a_2^{-1}) - b^2q_{\infty}(b^{-1})}{a_2-b},\\
    &\frac{(b-1)^2(a_1-1)(a_2-1)}{b^2a_1a_2}\,(AD-BC) \= b^2q_{\infty}(b^{-1}).    
    \end{split}
    \end{align}
In particular, these identities imply that $A$, $D$, and $BC$ 
are uniquely determined from the spectra. On the other hand, conjugation
by the diagonal matrix $D=\diag(\lambda,1,\lambda,1)$ preserves the shapes 
of $T_0$ and $T_1$ and maps $(B,C)$ to $(\lambda^{-1} B, \lambda C)$, hence
only the product $BC$ is uniquely determined.

\subsection{Criterion for irreducibility}
We now find a criterion for when the constructed representation 
is irreducible. The eigenmatrices for $T_0$ and $T_1$ are 
    \begin{equation*}
    Z_0 = \begin{pmatrix}
    1 & 0 & -A & -B \\
    0 & 1 & -C & -D \\
    0 & 0 & 1 & 0 \\
    0 & 0 & 0 & 1
    \end{pmatrix},
    \quad\quad
    Z_1 = \begin{pmatrix}
    0 & 0 & -1 & 0 \\
    0 & 0 & 0 & -1 \\
    1 & 0 & 1 & 0 \\
    0 & 1 & 0 & 1
    \end{pmatrix}.
    \end{equation*}

One can easily check the following assertions: if $C=0$, then the subspace of vectors of the form $(*,0,*,0)$ is invariant;
if $B=0$, then the subspace $(0,*,0,*)$ is invariant;
if $AD-BC=0$, then the subspace spanned by $\ker(T_1-I)$ and
the vector $(a,c,0,0)$ is invariant;
if $AD-BC-A-D+1=0$, then the subspace spanned by $\ker(T_1-bI)$
and the vector $(a-1,c,0,0)$ is invariant. Conversely, if 
$V$ is a nontrivial invariant subspace, then considering the
various possibilities for $V$ with respect to the eigenspaces
of $T_0$, we find that one of $B$, $C$, $AD-BC$, 
or $AD-BC-A-D+1$ must vanish.

Thus the representation is irreducible if and only if 
	$$BC(AD-BC)(AD-BC-A-D+1)\ne 0.$$

        To get the description in terms of eigenvalues we use the
        following factorizations:
    \begin{align}
    \begin{split}
    \label{eq:detfactors}
    AD-BC \;\;=\;\; &\frac{a_1a_2(1-bc_1)(1-bc_2)(1-bc_3)(1-bc_4)}{(1-b)^2(1-a_1)(1-a_2)},\\
    BC \;\;=\;\; &\frac{ba_2^3\prod_{1\le i<j\le 4}(1-a_1bc_ic_j)}{(a_1-a_2)^2(1-b)^2(1-a_1)(1-a_2)},\\
    AD-BC-A-D+1 \;\;=\;\;  &\frac{(1-c_1)(1-c_2)(1-c_3)(1-c_4)}{c_1c_2c_3c_4(1-b)^2(1-a_1)(1-a_2)}.
    \end{split}
    \end{align}
    Note that in terms of $q_{\infty}$ this simply becomes
    $q_\infty(1)=0$, $q_{\infty}(b^{-1})=0$, and
    $w_2(q_{\infty})(a_1^{-1}b^{-1})=0$, where
    $w_2(q_{\infty})=\prod_{i<j}(T-c_ic_j)$ is the polynomial whose
    roots are products of all pairs of roots of $q_{\infty}$.  This
    description agrees with the conditions given in
    \cite[p.~10]{Roberts}.

To summarize,  let 
$$
\TT:=\{(a_1,a_2,b,c_1,\ldots,c_4)\,|\, a_1a_2b^2c_1\cdots c_4=1\}\subseteq
S^1\times\cdots\times S^1\sm \Delta
$$
be the space of eigenvalues (taken in the unit circle). Here $\Delta$
is the union of $a_1=1,a_2=1,a_1=a_2,b=1$ and $c_i=c_j$
for $1\leq i<j\leq 4$ guaranteeing that
$(1,1,a_1,a_2),(1,1,b,b),(c_1,\ldots,c_4)$ are the eigenvalues of a
G-II system. Define $\TT^\irr$ as the subset corresponding to
irreducible local systems.  Then we have
$$
\TT^\irr = \TT \sm
\{q_\infty(1)q_{\infty}(b^{-1})w_2(q_{\infty})(a_1^{-1}b^{-1})=0 \}.
$$

The conditions for irreducibility we found can also be obtained
from~\cite[Thm. 1.5]{Crawley-Boevey}. Indeed the required
decompositions of the real root corresponding to G-II are the
following (and their refinements). Let $i_1,\ldots,i_4$ be any
re-ordering of $1,\ldots,4$.

\bigskip

i)    $q_\infty(1)=0$
$$
\begin{array}{ccc|c}
1& a_1& a_2&1 \\
1&1 & b & b\\
c_{i_1}&c_{i_2}&c_{i_3}&c_{i_4}
\end{array},\qquad 
a_1a_2b^2 c_{i_1}c_{i_2}c_{i_3}=1,\qquad
c_{i_4}=1.
$$

ii)  $q_{\infty}(b^{-1})=0$
$$
\begin{array}{ccc|c}
1& a_1& 1 & a_2 \\
1&1 & b & b\\
c_{i_1}&c_{i_2}&c_{i_3}&c_{i_4}
\end{array},\qquad 
a_1a_2b c_{i_1}c_{i_2}c_{i_3}=1,\qquad
b c_{i_4}=1.
$$

iii) $w_2(q_{\infty})(a_1^{-1}b^{-1})=0$
$$
\begin{array}{cc|cc}
1& a_1& 1 & a_2 \\
1&b & 1 & b\\
c_{i_1}&c_{i_2}&c_{i_3}&c_{i_4}
\end{array},\qquad 
a_1b c_{i_1}c_{i_2}=1,\qquad
a_2b c_{i_3}c_{i_4}=1.
$$

\subsection{ Invariant Hermitian form}

Let $T_1,\ldots,T_k$ correspond to an irreducible local system. Assume
that there exists a non-zero Hermitian form $H$ invariant under
$\Gamma$.
\begin{equation}
\label{herm-condition}
{T_s}^{*}HT_s \= H, \qquad s=1,\ldots, k.
\end{equation}
Since $\ker(H)$ is invariant under all $T_s$ by irreducibility we get
that any such $H$ must be non-degenerate. This implies, in particular,
that $({T_s}^{*})^{-1}$ and $T_s$ are conjugate. Therefore the sets of
eigenvalues of $T_s$ are invariant under the map
$z\mapsto \bar z^{-1}$. This is certainly the case if the eigenvalues
are in the unit circle.

On the other hand, if  the eigenvalues of $T_s$
are invariant under the map $z\mapsto \bar z^{-1}$ for all $s$ 
then the $(T_s^*)^{-1}$ give another solution 
to~\eqref{prod-mat}. If our system is rigid then there 
exists $H$ satisfying~\eqref{herm-condition}. Up to
a possible scalar factor $H$ is a Hermitian form invariant 
under the geometric monodromy group.

The set $\TT^\irr$ has finitely many connected components. The
signature of $H$ is constant on these components as it is 
continuous with integer values. 
We may further break the symmetry and choose the 
exponents satisfying $\alpha_1<\alpha_2$ and
$\gamma_1<\gamma_2<\gamma_3<\gamma_4$.  Then there is a unique
connected component $\TT^\irr_+$ where $H$ is positive definite.

We can compute the invariant form explicitly starting 
from~\eqref{eq:goursat_form}. 
The equations~\eqref{qinf_abcd} imply in this case that
$A$, $D$, and $BC$ are real. After making a suitable 
conjugation for $(B,C)$, we may assume that $A,B,C,D$ are real 
numbers. The invariant Hermitian matrix is then
	\begin{equation}
    \label{eq:goursat_hermitian}
    H = (AD-BC)\begin{pmatrix}
    C(1-DE) & BCE & C(1-D) & BC \\
    BCE & B(1-AE) & BC & B(1-A) \\
    C(1-D) & BC & C(1-D) & BC \\
    BC & B(1-A) & BC & B(1-A)
    \end{pmatrix},
    \end{equation}
where $E=(A+D-1)/(AD-BC)$.
The determinant of $H$ is
    \[(BC)^2(AD-BC-A-D+1)^3(AD-BC)^3.\]

Using~\eqref{eq:goursat_hermitian} we can easily 
describe~$\TT^\irr_+$ in terms of the 
parameters $(A,D,t)$ where $t=BC$.
If we look at the connected components of the set 
$\RR^3\sm V'$, where $V'=\{(A,D,t)\;|\;t(t-AD)(t-(1-A)(1-D))=0\}$,
and compute the signature in each case, we find that $H$ is 
positive definite if and only if
	\begin{equation*}
	\begin{cases}
	0<A,D<1,\\
	0<BC<AD,\\
	0<BC<(1-A)(1-D).
	\end{cases}
	\end{equation*}
    	
To derive a criterion in terms of eigenvalues requires more
work, but can be done similarly. The final criterion is then 
the following. Let $\mathcal{I}_1$ be the open arc in $S^1$ 
with end points $0$, and $b^{-1}$ (any of the two 
possibilities), and let $\mathcal{I}_2$ be the arc with end 
points $(b^{-1}a_1^{-1},b^{-1}a_2^{-1})$, where among the two 
arcs we pick the one that contains the point $b^{-1}$.
\begin{proposition}
	\label{hermposdef}
	The invariant Hermitian form $H$ is definite if and only
	if for some labeling $c_1,\dots,c_4$ of the eigenvalues 
	of $T_{\infty}$ we have 
	\begin{align*}
	&\textrm{(i)}\quad
	c_1,c_2 \in \mathcal{I}_1,\quad c_3,c_4\not\in\mathcal{I}_1,\\
	&\textrm{(ii)}\quad
	c_1c_2, c_3c_4, c_1c_3, c_2c_4 \in \mathcal{I}_2,\quad c_1c_4,c_2c_3\not\in\mathcal{I}_2.
	\end{align*}
\end{proposition}

\subsection{Integrality}
\label{integrality}
We would like to give an integral form of our local monodromies. The
first observation is that we may choose $T_\infty$ as the companion
form of $q_\infty$ since it has no repeated roots. After some
experimentation we found we can choose $T_1$ as follows.
    \[T_1=
    \left(
    \begin{array}{rrrr}
   \sigma_1&      0&0&         \sigma_2^{-1}a_1^{-1}\\
   \sigma_2(a_1+1)   &\sigma_1&1&  \sigma_1\sigma_2^{-1}{a_1^{-1}}\\
    -\sigma_1\sigma_2a_1    &-\sigma_2&0&                   1 + a_1^{-1}\\
    -\sigma_2^2a_1&      0&0&                           0\\
    
    \end{array}
    \right),\quad
    T_\infty=\left(
    \begin{array}{rrrr}
        0&0&0&-\tau_{4}\\
        1&0&0&\tau_{3}\\
        0&1&0&-\tau_{2}\\
        0&0&1&\tau_{1}\\
    \end{array}
    \right),\]
where $\sigma_i=e_i(b_1,b_2), \tau_i=e_i(c_1,c_2,c_3,c_4)$ are the
elementary symmetric functions.

With these, using that
$$
a_1a_2\sigma_2^2\tau_4=1,
$$
obtained by taking determinants in $T_0T_1T_\infty=I_4$, we get
$$
  T_0=\left(
    \begin{array}{rrrr}
(a_1 + 1)& 0& -\sigma_2^{-1}& a_2(\sigma_1\tau_4-\tau_3)\\
-\sigma_1a_1& 1& \sigma_1\sigma_2^{-1}&
                                        a_2(\tau_2-\sigma_2\tau_4)-\sigma_2^{-1}\\ 
 \sigma_2a_1& 0& 0&-a_2\tau_1+\sigma_1\sigma_2^{-1}\\
0& 0& 0&a_2
    \end{array}
    \right).
$$

The trace field is generically given by
$K=\QQ(\sigma_1,\sigma_2,\tau_1,\ldots,\tau_4)$ and we see that we can
always take as field of definition the quadratic extension
$F:=K(a_1)$.  Note that we also have $\tr(T_0)=2+a_1+a_2\in K$.  Hence
$a_1$ and $a_2$ are conjugate over $K$.

In fact, the local monodromies are definable over the ring $R[a_1]$,
where
$R:=\ZZ[\sigma_1,\sigma_2,\tau_1,\ldots,
\tau_4,\sigma_2^{-1},\tau_4^{-1}]$ and hence the group $\Gamma$ they
generate as well. The traces of all elements of the geometric
monodromy group are in $R$.

In particular, in the main case
of interest the characteristic polynomials $q_0,q_1,q_\infty$ will
have only roots of unity as roots. In this case $K$ is a cyclotomic
field. We conclude that the geometric monodromy can be 
conjugated to lie in $\GL_4(\calO_F)$, where $\calO_F$ is the ring of
integers of $F=K(a_1)$. This is consistent with the rigid local system being
motivic. 

For example, consider
$q_0=(x-1)^2(x^2+1),q_1=(x^2-1)^2,q_\infty= x^4 + (\zeta_{12}^3 -
\zeta_{12})x^3 - \zeta_{12}x + 1$, where $\zeta_{12}$ is a primitive
$12$-root of unity. This corresponds to row~$\# 3$ in
Table~\ref{special-1/4}. Then our choice gives
$$
T_0:=\left(
\begin{array}{cccc}
\zeta_{12}^3 + 1 &0&1& \zeta_{12}^2 - 1 \\
 0&1&0& -\zeta_{12}^3
 + 1 \\
 -\zeta_{12}^3&0&0&\zeta_{12}^2\\\
 0&0&0&-\zeta_{12}^3
\end{array}
\right),
\quad
T_1:=\left(
\begin{array}{cccc}
 0&0&0&\zeta_{12}^3\\
  \zeta_{12}^3
 + 1 &0&1&0\\
 0&1&0& -\zeta_{12}^3
 + 1 \\
 -\zeta_{12}^3&0&0&0\\
 \end{array}
\right),
$$
and
$$
T_\infty:=\left(
\begin{array}{cccc}
 0&0&0&-1\\
 1&0&0&\zeta_{12}\\
 0&1&0&0\\
 0&0&1&-\zeta_{12}^3
 + \zeta_{12}\\
\end{array}
\right)
$$
These matrices generate a group $\Gamma$ of order $103680$, which is a
non-split central extension by $C_4$ of the simple group
$\Gamma_{25920}$. We see here a phenomenon that occurs frequently in
our examples. The quotient $\Gamma/Z(\Gamma)$ has no irreducible
representation of degree $4$ (the smallest non-trivial irreducible
representation is of order $5$), whereas a central extension, namely
$\Gamma$, does.

It follows from the above discussion that for G-II cases the cocycle
obstruction of~\S\ref{cocycle} is generically of order dividing
$2=[F:K]$. We can easily compute the corresponding matrix $X_\sigma$
for $\sigma$ the generator of $\Gal(F/K)$ as in~\S\ref{cocycle}. The
problem is linear: we solve $T_sX_\sigma=X_\sigma T_s^\sigma$
generically, where $\sigma(a_1)=a_2$.  We find
$$
X_\sigma X_\sigma^\sigma=\mu I_4,
$$
where
$\mu=-(a_1\sigma_2)^3w_2(q_\infty)\left(a_1^{-1}\sigma_2^{-1}\right)\in
K^\times$. Recall that $w_2(q_{\infty}):=\prod_{i<j}(T-c_ic_j)$.
The cocycle can be represented by a quaternion algebra. Explicitly,
this is the quaternion algebra $\left(\frac{D,\mu}K\right)$, where
$D=\disc(F)$ and $\mu$ is as above.

\section{Finite monodromy}
\label{goursat-pos-defn}
We would like to describe all cases of G-II with finite geometric
monodromy. Since the geometric monodromy is integral~\S\ref{integrality}
finite monodromy is equivalent to the invariant Hermitian form being
definite in every complex embedding of the field of definition. (This
is the same argument used in~\cite{Beukers-Heckman}.) These cases are
those where all solutions to the corresponding differential
equation~\S\ref{diff-eqn} are algebraic.

Apart from the infinitely many imprimitive cases discussed later
in~\S\ref{G-II-inf}, the only examples of irreducible cases with finite
monodromy that we obtained after an extensive search are those given
in Tables~\ref{special-1/3},~\ref{special-1/4},~\ref{special-1/5}
and~\ref{general} below.

\subsection{Description of the tables}
\label{descr-tables}
For each choice of eigenvalues we list the order of the geometric
monodromy $\Gamma\subseteq \GL_4(\CC)$, an identification of $A$ and
the quotient of $\Gamma/A$ using standard notation ($A$ denotes a
maximal abelian normal subgroup of $\Gamma$), the order of the center
of $\Gamma$ and whether $\Gamma$ acts primitively or not.

By a theorem of Jordan there are finitely many possibilities for the
quotient $\Gamma/A$. The finite groups acting in four dimensions were
classified by Blichfeldt (see~\cite{Hanany-He} for a modern
description). The group denoted by $\Gamma_{25920}$ is a simple
group.

We should note that we can always twist the local monodromies by
multiplying by scalars matrices so that the resulting triple is in
$\SL_4(\CC)$. If the group acts primitively, the normal subgroup $A$
consists of scalars. It follows that there are finitely many possible
primitive $\Gamma$ up to twisting; we will see in~\S\ref{inf-families}
that this is not the case for imprimitive groups.

\subsection{Special case}
We start by discussing a special, simpler case.
Assume that the characteristic polynomials $q_0,q_1,q_\infty$ of
the local monodromies at the respective singularities have real
coefficients and that $q_1=(T-1)^2(T+1)^2$.  Let $\gamma_i\in (0,1)$ for
$i=1,\ldots, 4$ be the exponents of the roots of $q_\infty$ and
similarly  let $\alpha_1 \in (0,1/2)$ be such that the exponents of
$q_0$ are $0,0,\alpha_1,1-\alpha_1$.

A special case of Proposition~\ref{hermposdef} reduces in this case to
the following.  Let $\delta_1,\ldots,\delta_6$ be representatives in
$(0,1)$ (with multiplicities) of the exponents $\gamma_i+\gamma_j$ for
$i<j$.  Define $n_1$ as the number of $\gamma_i$ in the interval
$(0,1/2)$ and $n_2$ the number of $\delta_i$ (counting with
multiplicities) in the interval $(1/2-\alpha_1,1/2+\alpha_1)$.

\begin{proposition}
\label{goursat-pos-defn-1}
  With the above assumptions and notations the invariant Hermitian
  form $H$ is definite if and only if $(n_1,n_2)=(2,4)$.
\end{proposition}

To illustrate the situation here is a picture with the position of the
various roots on the unit circle in the case $\alpha_1,\alpha_2=1/3,2/3$ and
$\gamma=(1/28,9/28,3/4,25/28)$. Hence
$\delta=(1/14,3/14,5/14,9/14,11/14,13/14)$. 

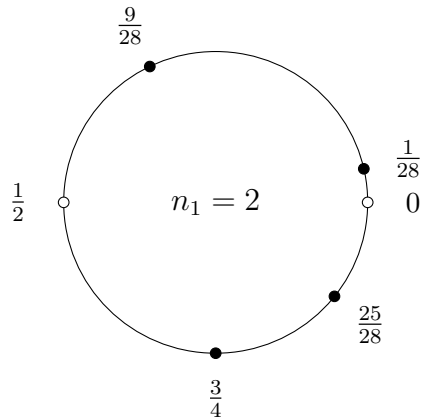
\begin{center}
\begin{tikzpicture}
\draw (0,0) circle (2);
\draw[fill=white] (2,0) circle (.07);
\node at (1.3*2,0) {\small $0$};
\draw[fill=white] (-2,0) circle (.07);
\node at (-1.3*2,0) {$\tfrac12$};
\draw[fill=black]  (1.949, 0.445) circle (.07);
\node at (1.3*1.949,1.3*.445) {$\tfrac 1{28}$};
\draw[fill=black] (-0.867,1.801) circle (.07);
\node at (-1.3*.867,1.3*1.801) {$\tfrac9{28}$};
\draw[fill=black] (0,-2) circle (.07);
\node at (0,-1.3*2) {$\tfrac34$};
\draw[fill=black] (1.563,- 1.246) circle (.07);
\node at (1.3*1.563,-1.3*1.246) {$\tfrac{25}{28}$};
\node at (0,0) {$n_1=2$};
\end{tikzpicture}
\end{center}

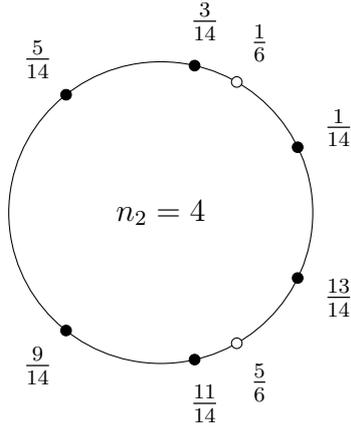
\begin{center}
\begin{tikzpicture}
\draw (0,0) circle (2);
\draw[fill=white] (1,-2*0.866) circle (.07);
\node at (1.3,-1.3*2*0.866) {$\tfrac56$};
\draw[fill=white] (1,2*0.866) circle (.07);
\node at (1.3,1.3*2*0.866) {$\tfrac16$};

\draw[fill=black]  (1.801,0.867) circle (.07);
\node at (1.3*1.801,1.3*0.867) {$\tfrac 1{14}$};
\draw[fill=black] (0.445,1.949) circle (.07);
\node at (1.3*.445,1.3*1.949) {$\tfrac3{14}$};
\draw[fill=black] (-1.246,1.563) circle (.07);
\node at (-1.3*1.246,1.3*1.563) {$\tfrac5{14}$};
\draw[fill=black]  (1.801,-0.867) circle (.07);
\node at (1.3*1.801,-1.3*0.867) {$\tfrac {13}{14}$};
\draw[fill=black] (0.445,-1.949) circle (.07);
\node at (1.3*.445,-1.3*1.949) {$\tfrac{11}{14}$};
\draw[fill=black] (-1.246,-1.563) circle (.07);
\node at (-1.3*1.246,-1.3*1.563) {$\tfrac9{14}$};

\node at (0,0) {$n_2=4$};
\end{tikzpicture}
\end{center}

In the special case of this section the finite monodromy cases found
are listed in
Tables~\ref{special-1/3},\ref{special-1/4},\ref{special-1/5} below.

\begin{table}[H]
$$
\begin{array}{|c|c|c|c|c|c|c|}
\hline
& \gamma & |\Gamma|&\Gamma/A&A&|Z(\Gamma)|&\text{Impr}\\
\hline
1&1/8, 3/8, 5/8, 7/8 &  48&S_4&C_2&2&*\\
\hline
2&1/5, 2/5, 3/5, 4/5 &  60&A_5&1&1&\\
\hline
3&1/10,3/10,7/10,9/10& 120&A_5&C_2&2&\\
\hline
4&1/12,5/12,7/12,11/12  & 144&C_2\times A_4&C_6&2&*\\
\hline
5&1/20,9/20,13/20,17/20  & 240 &A_5&C_4&4&\\
\hline
6&2/9,1/3,5/9,8/9&324&A_4&C_3^3&1&*\\
\hline
7&1/24, 7/24, 17/24, 23/24&576&S_4\times A_4&C_2&2&\\
\hline
8&1/28,9/28,3/4,25/28&672&\PSL_2(\FF_7)&C_4&4&\\
\hline
9&1/20,9/20,11/20,19/20&720&C_2\times A_5&C_6&2&*\\
\hline
10 & 1/15, 4/15, 11/15, 14/15&1440&A_5\times A_4&C_2&2&\\
\hline
11&1/30, 11/30, 19/30, 29/30&1440&A_5\times A_4&C_2&2&\\
\hline
12&1/40,9/40,31/40,39/40&2880&A_5\times S_4&C_2&2&\\
\hline
\end{array}
$$
\caption{$\alpha_1, \alpha_2=1/3,2/3$}
\label{special-1/3}
\end{table}

\begin{table}[H]
$$
\begin{array}{|c|c|c|c|c|c|c|}
\hline
& \gamma &  |\Gamma|&\Gamma/A&A&|Z(\Gamma)|&\text{Impr}\\
\hline
1&1/12,5/12,7/12,11/12  & 192&C_2\times S_4&C_4&2&*\\
\hline
2&1/20,9/20,13/20,17/20  & 640 &C_2^4\rtimes D_5&C_4&4&\\
\hline
3&1/36,13/36,25/36,11/12&103680&\Gamma_{25920}&C_4&4&\\
\hline
\end{array}
$$
\caption{$\alpha_1, \alpha_2=1/4,3/4$}
\label{special-1/4}
\end{table}

\begin{table}[H]
$$
\begin{array}{|c|c|c|c|c|c|c|}
\hline
& \gamma & |\Gamma|&\Gamma/A&A&|Z(\Gamma)|&\text{Impr}\\
\hline
1&1/12,5/12,7/12,11/12 & 1200&C_2\times A_5&C_{10}&2&*\\
\hline
2&2/15,7/15,8/15,13/15  & 7200&A_5\times A_5&C_2&2&\\
\hline
3&1/20,9/20,11/20,19/20  & 1200 &C_2\times A_5&C_{10}&2&*\\
\hline
4&1/30,11/30,19/30,29/30&7200&A_5\times A_5&C_2&2&\\
\hline
\end{array}
$$
\caption{$\alpha_1, \alpha_2=1/5,4/5$}
\label{special-1/5}
\end{table}

\subsection{General case}

Here we consider the general case (up to twisting) where the exponents
are
$$
\begin{array}{c|c}
t&{\rm exponents}\\
\hline
0& 0,0,\alpha_1,\alpha_2\\
1& 0,0,\beta,\beta\\
\infty&\gamma_1,\gamma_2,\gamma_3,\gamma_4
\end{array}
$$
The finite monodromy cases found are listed in Table~\ref{general}.

\begin{table}[H]
\footnotesize
\renewcommand{\arraystretch}{1.2}
\[
\begin{array}{|c|c|c|c|c|c|c|c|}
\hline
& \beta,\alpha_1,\alpha_2& \gamma & |\Gamma|&\Gamma/A&A&|Z(\Gamma)|&\text{Impr}\\
\hline
1&1/2,1/2,1/3 & 1/8, 11/24, 5/8, 23/24 &4608&(A_4\times A_4)\rtimes C_2&C_4^2&4&*\\
\hline
2&1/2,1/2,1/3 & 5/48, 23/48, 29/48, 47/48 &41472&(A_4\times A_4)\rtimes D_4&C_6^2&6&*\\
\hline
3&1/2,1/2,1/3 & 11/120, 59/120, 71/120, 119/120 &1036800&(A_5\times A_5)\rtimes C_2&C_{12}^2&12&*\\
\hline
4&1/2,1/2,1/4 & 7/48, 23/48, 31/48, 47/48 &6144& C_2^4\rtimes D_6&C_4\cdot C_8&8&*\\
\hline
5&1/2,1/2,1/5 & 7/40, 19/40, 27/40, 39/40 &2880000&(A_5\times A_5)\rtimes C_2&C_{20}^2&20&*\\
\hline
6&1/2,1/2,1/5 & 19/120, 59/120, 79/120, 119/120 &2880000&(A_5\times A_5)\rtimes C_2&C_{20}^2&20&*\\
\hline
7&1/2,1/2,1/6 & 5/36, 17/36, 29/36, 11/12 &311040& \Gamma_{25920}&C_{12}&12& \\
\hline
8&1/2,1/2,1/6 & 11/60, 23/60, 47/60, 59/60 &311040& \Gamma_{25920}&C_{12}&12& \\
\hline
9&1/2,1/3,1/4 & 7/24, 5/12, 19/24, 11/12 &165888&(A_4\times A_4)\rtimes D_4&C_{12}^2&12&*\\
\hline
10&1/2,1/3,1/4 & 11/48, 23/48, 35/48, 47/48 &165888&(A_4\times A_4)\rtimes D_4&C_{12}^2&12&*\\
\hline
11&1/2,1/3,1/5 & 4/15, 7/15, 23/30, 29/30 &6480000&(A_5\times A_5)\rtimes C_2&C_{30}^2&30&*\\
\hline
12&1/2,1/3,1/5 & 17/60, 9/20, 47/60, 19/20 &6480000&(A_5\times A_5)\rtimes C_2&C_{30}^2&30&*\\
\hline
13&1/2,1/3,1/5 & 19/60, 5/12, 49/60, 11/12 &6480000&(A_5\times A_5)\rtimes C_2&C_{30}^2&30&*\\
\hline
14&1/2,1/3,1/5 & 29/120, 59/120, 89/120, 119/120 &6480000&(A_5\times A_5)\rtimes C_2&C_{30}^2&30&*\\
\hline
15&1/2,1/5,2/5 & 4/15, 13/30, 23/30, 14/15 &6000&S_5&C_2\cdot C_{5}^2&10&*\\
\hline
16&1/2,1/5,2/5 & 9/40, 19/40, 29/40, 39/40 &6000&S_5&C_2\cdot C_{5}^2&10&*\\
\hline
17&1/3,1/2,5/6 & 1/18, 7/18, 13/18, 5/6 &155520&\Gamma_{25920}&C_6&6&\\
\hline
18&1/3,1/2,1/6 & 5/18, 11/18, 5/6, 17/18 &155520&\Gamma_{25920}&C_6&6&\\
\hline
19&1/3,1/2,1/6 & 11/30, 17/30, 23/30, 29/30 &155520&\Gamma_{25920}&C_6&6&\\
\hline
20&1/3,1/3,2/3 & 1/12, 11/24, 5/6, 23/24 &69120&C_2^4.A_6&C_{12}&12& \\
\hline
21&1/3,1/3,2/3 & 2/15, 8/15, 11/15, 14/15 &2160&A_6&C_6&6& \\
\hline
22&1/3,1/3,2/3 & 5/24, 11/24, 17/24, 23/24 &2160&A_6&C_6&6& \\
\hline
23&1/3,1/3,2/3 & 5/42, 17/42, 5/6, 41/42 &15120&A_7&C_6&6& \\
\hline
24&1/3,1/3,2/3 & 11/60, 23/60, 47/60, 59/60 &69120& C_2^4.A_6&C_{12}&12& \\
\hline
\end{array}
\]
\caption{General case}
\label{general}
\end{table}

\section{Coxeter groups}
\label{coxeter}

We may start with a finite group in $\GL_4(\CC)$ and attempt to build
a G-II rigid local system by producing three appropriate elements
$T_0,T_1,T_\infty$.  For example, we can take a finite complex
reflection group $W$ in rank $4$, hence one of the Weyl groups
$A_4,B_4,F_4$ or the non-crystallographic case $H_4$.
Since $T_\infty$ should have distinct eigenvalues different from~$1$,
we could start by taking $T_\infty$ to be a Coxeter
element. Similarly, we can take $T_1$ to be the product of two
commuting reflections in $W$. We may assume that these reflections are
simple and hence correspond to two non-adjacent dots in the
corresponding Dynkin diagram.

 We illustrate the procedure in one example in the case $H_4$ and give
 in section~\S\ref{finite-ratnl} examples defined over $\QQ$.  The
 Dynkin diagram is 

\begin{center}
  \begin{tikzpicture}
    \draw (-3,0) node[anchor=east]  {$H_4:$};
    \draw[thick,fill=black] (-2 ,0) circle(.07);
    \draw[thick,fill=black] (-1 ,0) circle(.07);
    \draw[thick,fill=black] (0,0) circle(.07);
    \draw[thick,fill=black] (1,0) circle(.07);
    \draw[thick] (-2,0) -- node [above, fill=white]{$5$} (-1, 0);
    \draw[thick] (-1,0) -- (0,0);
    \draw[thick] (0, 0) -- (1, 0);

    \draw (-2,0) circle (.12);
    \draw (0,0) circle (.12);
  \end{tikzpicture}
\end{center}
where we have circled the two chosen simple reflections. We take
$T_\infty:=s_1s_2s_3s_4, T_1:=s_1s_3$ and $T_0$ so that
$T_0T_1T_\infty=1$. 

With the help of MAGMA (see the actual calculations below) we find
that
$$
 T_0=
\left(
\begin{array}{rrrr}
 1&\tau&\tau&0\\
 0&0 &-1&0\\
 0&0&1&1\\
 0 &-1 &-1 &-1\\
\end{array}
\right),\quad
T_1=
\left(
\begin{array}{rrrr}
-1&0&0&0\\
 \tau&1&1&0\\
 0&0& -1&0\\
 0&0&1&1\\
\end{array}
\right),\quad
T_\infty=
\left(\begin{array}{rrrr}
-1&-\tau&-\tau&-\tau\\
\tau&\tau&\tau&\tau\\
0&1&0&0\\
0&0&1& 0\\
\end{array}\right),
$$
where $\tau^2-\tau-1=0$.

Taking the embedding into $\RR$ where $\tau=(1+\sqrt
5)/2=1.618033988\cdots$ is the golden ratio, the exponents of the
local monodromies are: 
	$$
	T_0: (0, 0, 1/3, 2/3), \quad T_1: (0, 0, 1/2, 1/2) ,\quad T_\infty:
	(1/30, 11/30, 19/30, 29/30)
	$$
and hence this example corresponds to row~$\# 11$ in the above table for
$\alpha_1, \alpha_2=1/3,2/3$. This shows, in particular, that here the monodromy
representation can be realized over $\ZZ[\tau]$, the ring of integers
of the trace field $K:=\QQ(\sqrt5)$. This is consistent with our
discussion in~\S\ref{integrality}.

\begin{verbatim}
> W<s1,s2,s3,s4>:=CoxeterGroup(GrpMat,"H4");
> K<a>:=BaseRing(W);
> K;
Number Field with defining polynomial x^2 - x - 1 over the Rational Field
> R<x>:=PolynomialRing(K);
> Tinf:=s1*s2*s3*s4;
> T1:=s1*s3;
> T0:=Tinf^(-1)*T1^(-1);
> CharacteristicPolynomial(Tinf);
x^4 + (-a + 1)*x^3 + (-a + 1)*x^2 + (-a + 1)*x + 1
> CharacteristicPolynomial(T1);
x^4 - 2*x^2 + 1
> CharacteristicPolynomial(T0);
x^4 - x^3 - x + 1
> T0;
[ 1  a  a  0]
[ 0  0 -1  0]
[ 0  0  1  1]
[ 0 -1 -1 -1]
> T1;
[-1  0  0  0]
[ a  1  1  0]
[ 0  0 -1  0]
[ 0  0  1  1]
> Tinf;
[-1 -a -a -a]
[ a  a  a  a]
[ 0  1  0  0]
[ 0  0  1  0]
> G:=sub<W|[T1,Tinf]>; 
> #G;
1440
\end{verbatim}
\section{Differential equation}
\label{diff-eqn}

Goursat~\cite[\S 10]{Goursat} computes explicitly an order four linear
differential equation of type G-II with given local
monodromies. Let the exponents of these local monodromies be
$$
0: (0,1,1-\alpha_1,1-\alpha_2), \quad
1: (0,1,\beta,\beta+1), \quad
\infty: (\gamma_1,\gamma_2,\gamma_3,\gamma_4).
$$
These are  are related
by the equation 
$$
\beta=\tfrac12(1+e_1(\alpha)-e_1(\gamma)), \qquad
\alpha:=(\alpha_1,\alpha_2),\qquad \gamma:=(\gamma_1,\ldots,\gamma_4),
$$
where $e_n(x)$ denotes the elementary symmetric functions of
the quantities $x=(x_1,x_2,\ldots)$. 

The shape of the exponents forces the differential equation to be of
the following form
$$
 x^2(x-1)^2\frac{d^4}{dx^4}
        +(Ax-B)x(x-1)\frac{d^3}{dx^3}
	+(Cx^2-Dx+E)\frac{d^2}{dx^2}
	+(Fx-G)\frac d{dx}
        +H=0.
$$
for certain constants $A,B,\ldots,H$. Imposing further that there are
no logarithmic solutions at $x=1$ completely determines these
constants. Following Goursat we obtain the following values (the
expression for $G$ he gives is not quite right and is here corrected).
\begin{align*}
A&=6+e_1(\gamma)\\
B&=3+e_1(\alpha)\\
C&=7+3e_1(\gamma)+e_2(\gamma)\\
D&=E+C-(\beta-1)(\beta-2)\\
E&=1+e_1(\alpha)+e_2(\alpha)\\
F&=1+e_1(\gamma)+e_2(\gamma)+e_3(\gamma)\\
G&=F+2(\beta-1)(\beta-2)(\beta-3)+(\beta-1)(\beta-2)(2A-B)
                +(\beta-1)(2C-D)
\\
H&=e_4(\gamma).
\end{align*}

As worked out also by Goursat~\cite[(18)]{Goursat} the 
coefficients $a_n$ of a power series solution $\sum_{n\geq 0}a_nx^n$ 
to the differential equation satisfy the second order recursion
\begin{equation}
\label{recursion}
C_0a_n=C_1a_{n-1}+C_2a_{n-2},
\end{equation}
where
\begin{align*}
C_0&:=(n+1)(n+2)[n(n-1)+Bn+E]\\
C_1&:=(n+1)[2n(n-1)(n-2) +(A+B)n(n-1) +Dn+G]\\
C_2&:= -[n(n-1)(n-2)(n-3) +An(n-1)(n-2) +Cn(n-1) +Fn +H]
\end{align*}

For example, if we take $\alpha=(1/4,3/4)$ and
$\gamma=(1/5,2/5,3/5,-1/5)$ (so that $\beta=1/2$) we obtain the
differential equation
\begin{multline}
x^2(x-1)^2\frac{d^4}{dx^4}
+x(x-1) (7x- 4)\frac{d^3}{dx^3}\\
+(51/5x^2 - 931/80x + 35/16)\frac{d^2}{dx^2}
+(54/25x - 1223/800)\frac{d}{dx}
 -6/625,
\end{multline}
which indeed has  local exponents
$$
0: (0,1,1/4,3/4), \quad
1: (0, 1/2, 1, 3/2), \quad
\infty: (-1/5, 1/5, 2/5, 3/5).
$$
The first few coefficients of the power series expansion  of a
basis of holomorphic solutions to the differential equation at $x=0$
are as follows
\begin{align}
\begin{split}
\phi_0&=1 + 48/21875x^2 + 28088/18046875x^3 + 6589643/5865234375x^4\\
& +57582020413/67659667968750x^5 + O(x^6)\\
\phi_1&= x + 1223/3500x^2 +
1096811/5775000x^3 + 370276451/3003000000x^4\\
& + 15278570717561/173208750000000x^5 +
O(x^6)
\end{split}
\end{align}

As expected from a motivic situation the denominators of the
coefficients appear to grow only exponentially, rather than what could
be expected generically from solutions to the
recursion~\eqref{recursion}. 

\section{Field of moduli $\QQ$}

It is easy to list all cases of irreducible G-II rigid local systems
with field of moduli $\QQ$ as there are only finitely many cyclotomic
polynomials of fixed degree and with coefficients in $\QQ$. The
resulting cases are listed in Table~\ref{ratnl-indef} and
Table~\ref{ratnl-defn} according to the signature of the respective
Hermitian form.

\begin{table}
$$
\begin{array}{|r|r|r|c|r|}
\hline
&\alpha_1,\alpha_2&\beta_1,\beta_2&\gamma_1,\gamma_2,\gamma_3,\gamma_4&\mu\\
\hline
1&1/3, 2/3 & 0, 1/2 & 1/4, 1/3, 2/3, 3/4 & -2\\
\hline
2&1/3, 2/3 & 0, 1/2 & 1/6, 1/4, 3/4, 5/6 & -2\\
\hline
3&1/3, 2/3 & 1/3, 2/3 & 1/6, 1/4, 3/4, 5/6 & -2\\
\hline
4&1/3, 2/3 & 1/3, 2/3 & 1/10, 3/10, 7/10, 9/10 & -1\\
\hline
5&1/3, 2/3 & 1/6, 5/6 & 1/4, 1/3, 2/3, 3/4 & -2\\
\hline
6&1/3, 2/3 & 1/6, 5/6 & 1/5, 2/5, 3/5, 4/5 & -1\\
\hline
7&1/4, 3/4 & 0, 1/2 & 1/4, 1/3, 2/3, 3/4 & -3\\
\hline
8&1/4, 3/4 & 0, 1/2 & 1/6, 1/3, 2/3, 5/6 & -1\\
\hline
9&1/4, 3/4 & 0, 1/2 & 1/6, 1/4, 3/4, 5/6 & -3\\
\hline
10&1/4, 3/4 & 0, 1/2 & 1/5, 2/5, 3/5, 4/5 & -1\\
\hline
11&1/4, 3/4 & 0, 1/2 & 1/10, 3/10, 7/10, 9/10 & -1\\
\hline
12&1/4, 3/4 & 1/3, 2/3 & 1/6, 1/4, 3/4, 5/6 & -3\\
\hline
13&1/4, 3/4 & 1/3, 2/3 & 1/10, 3/10, 7/10, 9/10 & -1\\
\hline
14&1/4, 3/4 & 1/3, 2/3 & 1/12, 5/12, 7/12, 11/12 & 1\\
\hline
15&1/4, 3/4 & 1/4, 3/4 & 1/12, 5/12, 7/12, 11/12 & 1\\
\hline
16&1/4, 3/4 & 1/6, 5/6 & 1/4, 1/3, 2/3, 3/4 & -3\\
\hline
17&1/4, 3/4 & 1/6, 5/6 & 1/5, 2/5, 3/5, 4/5 & -1\\
\hline
18&1/4, 3/4 & 1/6, 5/6 & 1/12, 5/12, 7/12, 11/12 & 1\\
\hline
19&1/6, 5/6 & 0, 1/2 & 1/4, 1/3, 2/3, 3/4 & -2\\
\hline
20&1/6, 5/6 & 0, 1/2 & 1/6, 1/3, 2/3, 5/6 & -2\\
\hline
21&1/6, 5/6 & 0, 1/2 & 1/6, 1/4, 3/4, 5/6 & -2\\
\hline
22&1/6, 5/6 & 0, 1/2 & 1/5, 2/5, 3/5, 4/5 & -1\\
\hline
23&1/6, 5/6 & 0, 1/2 & 1/8, 3/8, 5/8, 7/8 & -1\\
\hline
24&1/6, 5/6 & 0, 1/2 & 1/10, 3/10, 7/10, 9/10 & -1\\
\hline
25&1/6, 5/6 & 1/3, 2/3 & 1/6, 1/4, 3/4, 5/6 & -2\\
\hline
26&1/6, 5/6 & 1/3, 2/3 & 1/5, 2/5, 3/5, 4/5 & 1\\
\hline
27&1/6, 5/6 & 1/3, 2/3 & 1/8, 3/8, 5/8, 7/8 & 1\\
\hline
28&1/6, 5/6 & 1/3, 2/3 & 1/10, 3/10, 7/10, 9/10 & 1\\
\hline
29&1/6, 5/6 & 1/3, 2/3 & 1/12, 5/12, 7/12, 11/12 & 2\\
\hline
30&1/6, 5/6 & 1/4, 3/4 & 1/5, 2/5, 3/5, 4/5 & 1\\
\hline
31&1/6, 5/6 & 1/4, 3/4 & 1/8, 3/8, 5/8, 7/8 & 1\\
\hline
32&1/6, 5/6 & 1/4, 3/4 & 1/10, 3/10, 7/10, 9/10 & 1\\
\hline
33&1/6, 5/6 & 1/4, 3/4 & 1/12, 5/12, 7/12, 11/12 & 2\\
\hline
34&1/6, 5/6 & 1/6, 5/6 & 1/4, 1/3, 2/3, 3/4 & -2\\
\hline
35&1/6, 5/6 & 1/6, 5/6 & 1/5, 2/5, 3/5, 4/5 & 1\\
\hline
36&1/6, 5/6 & 1/6, 5/6 & 1/8, 3/8, 5/8, 7/8 & 1\\
\hline
37&1/6, 5/6 & 1/6, 5/6 & 1/10, 3/10, 7/10, 9/10 & 1\\
\hline
38&1/6, 5/6 & 1/6, 5/6 & 1/12, 5/12, 7/12, 11/12 & 2\\
\hline
\end{array}
$$
\caption{Signature $(2,2)$}
\label{ratnl-indef}
\end{table}

\begin{table}
$$
\begin{array}{|r|r|r|c|r|}
\hline
&\alpha_1,\alpha_2&\beta_1,\beta_2&\gamma_1,\gamma_2,\gamma_3,\gamma_4&\mu\\
\hline
1&1/3, 2/3 & 0, 1/2 & 1/5, 2/5, 3/5, 4/5 &1\\
\hline
2&1/3, 2/3 & 0, 1/2 & 1/8, 3/8, 5/8, 7/8 &1\\
\hline
3&1/3, 2/3 & 0, 1/2 & 1/10, 3/10, 7/10, 9/10 &1\\
\hline
4&1/3, 2/3 & 0, 1/2 & 1/12, 5/12, 7/12, 11/12 &2\\
\hline
5&1/3, 2/3 & 1/3, 2/3 & 1/5, 2/5, 3/5, 4/5 &-1\\
\hline
6&1/3, 2/3 & 1/3, 2/3 & 1/8, 3/8, 5/8, 7/8 &-1\\
\hline
7&1/3, 2/3 & 1/4, 3/4 & 1/6, 1/3, 2/3, 5/6 &-2\\
\hline
8&1/3, 2/3 & 1/4, 3/4 & 1/5, 2/5, 3/5, 4/5 &-1\\
\hline
9&1/3, 2/3 & 1/4, 3/4 & 1/8, 3/8, 5/8, 7/8 &-1\\
\hline
10&1/3, 2/3 & 1/4, 3/4 & 1/10, 3/10, 7/10, 9/10 &-1\\
\hline
11&1/3, 2/3 & 1/6, 5/6 & 1/8, 3/8, 5/8, 7/8 &-1\\
\hline
12&1/3, 2/3 & 1/6, 5/6 & 1/10, 3/10, 7/10, 9/10 &-1\\
\hline
13&1/4, 3/4 & 0, 1/2 & 1/12, 5/12, 7/12, 11/12 &1\\
\hline
14&1/4, 3/4 & 1/3, 2/3 & 1/5, 2/5, 3/5, 4/5 &-1\\
\hline
15&1/4, 3/4 & 1/4, 3/4 & 1/6, 1/3, 2/3, 5/6 &-1\\
\hline
16&1/4, 3/4 & 1/4, 3/4 & 1/5, 2/5, 3/5, 4/5 &-1\\
\hline
17&1/4, 3/4 & 1/4, 3/4 & 1/10, 3/10, 7/10, 9/10 &-1\\
\hline
18&1/4, 3/4 & 1/6, 5/6 & 1/10, 3/10, 7/10, 9/10 &-1\\
\hline
\end{array}
$$
\caption{Signature $(4,0)$}
\label{ratnl-defn}
\end{table}

There are only four different quaternion algebras over $\QQ$ that
appear depending on $\alpha_1,\alpha_2$ and $\mu$.  To give these
algebras is enough to give the list $[p_1,\ldots,p_{2r}]$ of ramified
primes. These are: $[\ ]$ (i.e. the matrix algebra),
$[2,\infty],[3,\infty]$ or $[2,3]$.  We tabulate what we get below in
Table~\ref{quaternion-table} where $D=\disc(F)$. When $\mu=1$ the
obstruction is trivial.

\begin{table}
$$
\begin{array}{|r|r|r|r|r|}
\hline
D\backslash \mu & -3 & -2&-1& 2\\
\hline
-3& [3,\infty]&[2,\infty]&[3,\infty]&[2,3]\\
\hline
-4& [3,\infty]&[2,\infty]&[2,\infty]&\\
\hline
\end{array}
$$
\caption{Quaternion algebras}
\label{quaternion-table}
\end{table}

\section{Finite monodromy  $K=\QQ$}
\label{finite-ratnl}
As we see in the above table there are only four cases. Three cases
are actually definable over $\QQ$; we list these first. We give the
fourth case in ~\S\ref{finite-extra}; it has the quaternion algebra
ramified at $[2,3]$ as an obstruction and is hence not definable over
$\QQ$.

We will construct these examples in Coxeter groups as
in~\S\ref{coxeter}; we circle in the corresponding Dynkin diagram the two
chosen simple reflections.

\subsection{$(1/3,2/3),(0,1/2),(1/5,2/5,3/5,4/5)$}

We can find this case as a subgroup of $S_5$ viewed as the Coxeter
group of the root system~$A_4$. 

\begin{center}
  \begin{tikzpicture}
    \draw (-3,0) node[anchor=east]  {$A_4:$};
    \draw[thick,fill=black] (-2 ,0) circle(.07);
    \draw[thick,fill=black] (-1 ,0) circle(.07);
    \draw[thick,fill=black] (0,0) circle(.07);
    \draw[thick,fill=black] (1,0) circle(.07);
    \draw[thick] (-2,0) --  (-1, 0);
    \draw[thick] (-1,0) -- (0,0);
    \draw[thick] (0, 0) -- (1, 0);

    \draw (-2,0) circle (.12);
    \draw (0,0) circle (.12);
  \end{tikzpicture}
\end{center}

The geometric monodromy group is
isomorphic to $A_5$, the alternating group in five letters, acting in
its standard representation. Here is a
calculation using MAGMA.
\begin{verbatim}
> W<s1,s2,s3,s4>:=CoxeterGroup(GrpMat,"A4");  
> K:=BaseRing(W);
> R<x>:=PolynomialRing(K);
> Tinf:=s1*s2*s3*s4;
> T1:=s1*s3;
> T0:=Tinf^(-1)*T1^(-1);
> CharacteristicPolynomial(T0);
x^4 - x^3 - x + 1
> CharacteristicPolynomial(T1);
x^4 - 2*x^2 + 1
> CharacteristicPolynomial(Tinf);
x^4 + x^3 + x^2 + x + 1
> G:=sub<W|[T0,T1,Tinf]>;             
> IsIsomorphic(G,AlternatingGroup(5));
true Homomorphism of MatrixGroup(4, Rational Field) of order 2^2 * 3 * 5 into 
GrpPerm: $, Degree 5, Order 2^2 * 3 * 5 induced by
    [ 1  1  1  0]
    [ 0  0 -1  0]
    [ 0  0  1  1]
    [ 0 -1 -1 -1] |--> (2, 5, 3)
    [-1  0  0  0]
    [ 1  1  1  0]
    [ 0  0 -1  0]
    [ 0  0  1  1] |--> (1, 2)(4, 5)
    [-1 -1 -1 -1]
    [ 1  0  0  0]
    [ 0  1  0  0]
    [ 0  0  1  0] |--> (1, 3, 5, 4, 2)
\end{verbatim}
Choosing the parameters in Goursat's differential equation~\S
\ref{diff-eqn} as 
$$
(\alpha_1, \alpha_2)=(1/3,2/3);\quad
\beta=1/2;\quad
(\gamma_1,\gamma_2,\gamma_3,\gamma_4)=(1/5,2/5,-2/5,4/5)
$$
we obtain $7, 4, 10, 413/36, 20/9, 46/25,
2387/1800,-16/625$ for the eight constants
$A,B,\ldots,H$. Then all holomorphic solutions at
$x=0$ have power series expansion, since they represent algebraic
functions of
$x$, with integral coefficients, up to the power of some constant
$N$. (The minimal such
$N$ is called the Eisenstein constant of the algebraic function; in
this example it seems to only involve the primes $2,5$ and $11$.)

The holomorphic solution to equation holomorphic at $x=0$ and
starting as 
$$
y:=1 - 387/1300\,x - 172773/2080000\,x^2 - 141382989/3328000000\,x^3 +O(x^4)
$$
satisfies an algebraic equation of degree $10$ over
$\QQ(x)$.  (The series can be computed using the explicit form of the
differential equation given by Goursat, see~\S\ref{diff-eqn}.)

The solution over $K:=\QQ(\sqrt{-15})$ that starts
as
$$
 1 - (123/475 + 33/1900\, \omega )x - (271713/3800000 +
 78771/15200000\, \omega )x^2 +O(x^3)
$$
where $\, \omega ^2-\, \omega +4=0$ is a generator of the ring of
integers of $K$ on the other hand, satisfies the following degree five
equation
$$
P(x,y):=y^5+a_3(x)y^3+a_2(x)y^2+a_1(x)y+a_0(x)=0,
$$
where
\begin{align*}
\begin{split}
a_3&:=605/8664 - 715/2888\, \omega \\
a_2&:=-1189825/2963088 +
70525/329232\, \omega \\
a_1&:=(298150/390963 - 11050/130321\, \omega )x 
-518989705/900778752 + 19234735/300259584\, \omega \\
a_0&:= -(453252/2476099 +
151020/2476099\, \omega )x^2
+ (3663787/14856594 + 406915/4952198\, \omega )x\\
&-(82982887/900778752 + 9216415/300259584\, \omega )
\end{split}
\end{align*}

\subsection{$(1/4,3/4),(0,1/2),(1/8,3/8,5/8,7/8)$}
We can find this case as a subgroup of the Coxeter group of the root
system~$B_4$.

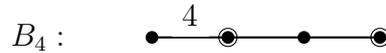
\begin{center}
  \begin{tikzpicture}
    \draw (-3,0) node[anchor=east]  {$B_4:$};
    \draw[thick,fill=black] (-2 ,0) circle(.07);
    \draw[thick,fill=black] (-1 ,0) circle(.07);
    \draw[thick,fill=black] (0,0) circle(.07);
    \draw[thick,fill=black] (1,0) circle(.07);
    \draw[thick] (-2,0) -- node [above, fill=white]{$4$} (-1, 0);
    \draw[thick] (-1,0) -- (0,0);
    \draw[thick] (0, 0) -- (1, 0);

    \draw (-1,0) circle (.12);
    \draw (1,0) circle (.12);
  \end{tikzpicture}
\end{center}

 The geometric monodromy group is isomorphic to
$\GL_2(\FF_3)$, of order $48$, in its unique faithful irreducible
representation of dimension four. Here is a calculation using MAGMA.
\begin{verbatim}
> W<s1,s2,s3,s4>:=CoxeterGroup(GrpMat,"B4");
> K:=BaseRing(W);
> R<x>:=PolynomialRing(K);
> Tinf:=s1*s2*s3*s4;
> T1:=s2*s4;
> T0:=Tinf^(-1)*T1^(-1);
> CharacteristicPolynomial(T0);
x^4 - x^3 - x + 1
> CharacteristicPolynomial(T1);
x^4 - 2*x^2 + 1
> CharacteristicPolynomial(Tinf);
x^4 + 1
> T0;
[ 0 -1  0  0]
[ 0  1  1  2]
[ 1  1  1  0]
[-1 -1 -1 -1]
> T1;
[ 1  1  0  0]
[ 0 -1  0  0]
[ 0  1  1  2]
[ 0  0  0 -1]
> Tinf;
[-1 -1 -1 -2]
[ 1  0  0  0]
[ 0  1  0  0]
[ 0  0  1  1]
> G:=sub<W|[T0,T1,Tinf]>;
> IsIsomorphic(G,GL(2,GF(3)));
true Mapping from: GrpMat: G to GL(2, GF(3))
Composition of Mapping from: GrpMat: G to GrpPC and
Mapping from: GrpPC to GrpPC and
Mapping from: GrpPC to GL(2, GF(3))
\end{verbatim}

We can change basis so that $T_\infty$ is the companion matrix of
$\Phi_8=T^4+1$. We obtain  the following triple
$$
T_0=
\left(
\begin{array}{rrrr}
 0&1&1&0\\
 -1&0&-1&0\\
 0&-1&0&0\\
 0&1&0&1\\
\end{array}
\right),
\quad
T_1=
\left(
\begin{array}{rrrr}
 -1&-1&-1&-1\\
 0&0&0&-1\\
 0&1&0&1\\
 -1&0&0&1\\
\end{array}
\right),
\quad
T_\infty=
\left(
\begin{array}{rrrr}
 0&0&0&-1\\
 1&0&0&0\\
 0&1&0&1\\
 0&0&1&0\\
\end{array}
\right)
$$
Choosing the parameters in Goursat's differential equation below~\S
\ref{diff-eqn} as 
$$
(\alpha_1, \alpha_2)=(1/4,3/4);\quad
\beta=1/2;\quad
(\gamma_1,\gamma_2,\gamma_3,\gamma_4)=(1/8,3/8,-3/8,7/8)
$$
we obtain $7, 4, 319/32, 3295/288, 20/9, 117/64, 383/288, -63/4096$
for the eight constants $A,B,\ldots,H$. Then all holomomorphic
solutions at $x=0$ have power series expansion with integral
coefficients up to powers of $2$ and $5$.

The holomorphic solution to equation holomorphic at $x=0$ and
starting as 
$$
1 + (5/256\sqrt{-8} - 29/128)x + (383/65536\sqrt{-8} - 527/8192)x^2 + O(x^3)
$$
satisfies the degree eight equation
$$
P(x,y)=y^8+a_6(x)y^6+a_4(x)y^4+a_2(x)y^2+a_0(x),
$$
where
\begin{align*}
\begin{split}
a_6&:=230/729\sqrt{-8} - 400/729\\
a_4&:= (1048/19683\sqrt{-8} + 19984/19683)x -(351670/1594323\sqrt{-8} +
1034482/1594323)\\ 
a_2&:= (4842880/43046721\sqrt{-8} -
10078688/43046721)x\\
& + (-1015591450/10460353203\sqrt{-8} +
1684358888/10460353203)\\
a_0&:=-(27028768/1162261467\sqrt{-8} + 3467632/1162261467)x^3 \\
&+
(172219360/3486784401\sqrt{-8} + 238769752/3486784401)x^2\\
& -(296048878/10460353203\sqrt{-8} + 1067187679/10460353203)x\\
&+
(22649710/10460353203\sqrt{-8} + 382087111/10460353203)
\end{split}
\end{align*}

In addition, let $\phi_0:=1+0\cdot x +O(x^2),\phi_1:=0\cdot 1 +x
+O(x^2)$ be a basis of the holomorphic solutions to the differential
equation at $x=0$ and define
$$
\psi:=\phi_0^2-891/16384\phi_1^2.
$$
Then $\psi$ satisfies a degree four equation; more precisely, if
$\xi$ is the hypergeometric series satisfying the trinomial
equation
$$
\xi^4-4\xi^3+27x=0
$$
then $\psi=-4/135\xi^3+4/45\xi^2+32/45\xi-37/27$.

\subsection{$(1/4,3/4),(0,1/2),(1/12,5/12,7/12,11/12)$}
We can find this case as a subgroup of the Coxeter group of the root
system~$F_4$. 

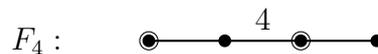
\begin{center}
  \begin{tikzpicture}
    \draw (-3,0) node[anchor=east]  {$F_4:$};
    \draw[thick,fill=black] (-2 ,0) circle(.07);
    \draw[thick,fill=black] (-1 ,0) circle(.07);
    \draw[thick,fill=black] (0,0) circle(.07);
    \draw[thick,fill=black] (1,0) circle(.07);
    \draw[thick] (-2,0) -- (-1, 0);
    \draw[thick] (-1,0) -- node [above, fill=white]{$4$} (0,0);
    \draw[thick] (0, 0) -- (1, 0);

    \draw (-2,0) circle (.12);
    \draw (0,0) circle (.12);
  \end{tikzpicture}
\end{center}

The geometric monodromy group is isomorphic to
$U_2(\FF_3)\rtimes C_2$, of order $192$, in one of its irreducible
faithful representations of dimension four.  This group is labeled
$(192,988)$ in the {\tt SmallGroup} database.  Here is a calculation
using MAGMA.
\begin{verbatim}
> W<s1,s2,s3,s4>:=CoxeterGroup(GrpMat,"F4");
> K:=BaseRing(W);
> R<x>:=PolynomialRing(K);
> Tinf:=s1*s2*s3*s4;
> T1:=s1*s3;                                 
> T0:=Tinf^(-1)*T1^(-1);                     
> CharacteristicPolynomial(T0);
x^4 - 2*x^3 + 2*x^2 - 2*x + 1
> CharacteristicPolynomial(T1);
x^4 - 2*x^2 + 1
> CharacteristicPolynomial(Tinf);
x^4 - x^2 + 1
> T0;
[ 1  1  2  0]
[ 0  1  0  0]
[ 0  0  1  1]
[ 0 -1 -2 -1]
> T1;
[-1  0  0  0]
[ 1  1  2  0]
[ 0  0 -1  0]
[ 0  0  1  1]
> Tinf;
[-1 -1 -2 -2]
[ 1  0  0  0]
[ 0  1  1  1]
[ 0  0  1  0]
> G:=sub<W|[T0,T1,Tinf]>;                    
> IdentifyGroup(G);
<192, 988>
\end{verbatim}
Conjugating so that $T_\infty$ is the companion matrix of $\Phi_{12}=T^4-T^2+1$
we obtain the following triple
$$
T_0=\left(\begin{array}{rrrr}
 1&0&0&-1\\
 0&1&0&1\\
 0&0&1&1\\
 0&-1&-1&-1\\
\end{array}
\right),
\quad
T_1=\left(
\begin{array}{rrrr}
 0&1&1&1\\
 1&0&0&-1\\
 0&0&-1&0\\
 0&0&1&1\\
\end{array}
\right),
\quad
T_\infty=\left(
\begin{array}{rrrr}
 0&0&0&-1\\
 1&0&0&0\\
 0&1&0&1\\
 0&0&1&0\\
\end{array}
\right)
$$
The permutation representation of smallest degree for this monodromy
group is of degree~$24$. Hence some solution to the differential
equation satisfies a degree $24$ equation with coefficients in  some
number field but we have not attempted to find it.

\subsection{$(1/3,2/3),(0,1/2),(1/12,5/12,7/12,11/12)$}
\label{finite-extra}
As already mentioned in this case the quaternion algebra is ramified
at $[2,3]$. As it happens, since $2$ and $3$ are both inert in
$F=\QQ(\sqrt 5)$, the obstruction cocycle $\xi$ becomes trivial in
$F$. Therefore by Proposition~\ref{cocycle} we should be able to
realize this case over $F$.

This indeed the case and we can realize it again using Coxeter groups;
namely, as a subgroup of the non-crystallographic $W(H_4)$ of order
$14400$.  Note, however, that in this case $T_\infty$ is not a Coxeter
element; we still take $T_1$ as a product of two commuting simple
reflections.

Here are some of the computations in MAGMA.
\begin{verbatim}
> W<s1,s2,s3,s4>:=CoxeterGroup(GrpMat,"H4");
> K<a>:=BaseRing(W);
> K;
Number Field with defining polynomial x^2 - x - 1 over the Rational Field
> R<x>:=PolynomialRing(K);
> CC:=ConjugacyClasses(W);
> [Order(g[3]): g in CC];
[ 1, 2, 2, 2, 2, 3, 3, 4, 4, 5, 5, 5, 5, 5, 6, 6, 6, 6, 10, 10, 10, 10, 10, 10, 
10, 10, 10, 12, 15, 15, 20, 20, 30, 30 ]
\end{verbatim}
We see that there is a unique conjugacy class in $W(H_4)$ of order
$12$. The class has $1200$ elements.
\begin{verbatim}
> C:=Conjugates(W,CC[28][3]);
> #C;
1200
\end{verbatim}
We set $T_1:=s_1s_3$ and look for an element $T_\infty$ of order $12$
such that  $s_1s_3T_\infty$ has characteristic polynomial
$T^4-T^3-T+1=(T-1)^2(T^2+T+1)$. There a fair number of such elements;
we select for example:
\begin{verbatim}
> Tinf;
[     0      1      1      1]
[ a + 1  a + 1      1      0]
[    -a -a - 1     -1      0]
[-a - 1 -a - 1     -a     -a]
\end{verbatim}
We can now construct the whole triple and compute the order of the
group they generate.
\begin{verbatim}
> T1:=s1*s3;
> T0:=Tinf^(-1)*T1^(-1);
> T0;
[      a       1       0       0]
[ -a - 1  -a - 1      -1      -1]
[2*a + 1 2*a + 2   a + 2   a + 1]
[ -a - 1  -a - 1  -a - 1      -a]
> T1;
[-1  0  0  0]
[ a  1  1  0]
[ 0  0 -1  0]
[ 0  0  1  1]
> CharacteristicPolynomial(T0);
x^4 - x^3 - x + 1
> CharacteristicPolynomial(T1);
x^4 - 2*x^2 + 1
> CharacteristicPolynomial(Tinf);
x^4 - x^2 + 1
> G:=sub<W|[T0,T1,Tinf]>;
> #G;
144
> IdentifyGroup(G);
<144, 127>
\end{verbatim}
The geometric monodromy group is then isomorphic to
$\SL_2(\FF_3)\rtimes S_3$ acting via its four dimensional faithful
irreducible representation with rational character. Note that this
representation has Schur index $2$ matching our  obstruction
calculation. Its smallest permutation representation is of degree
$48$. 

\section{Hurwitz example}

Consider the modular function $u_7(\tau):=\eta(7\tau)/\eta(\tau)$,
where $\eta(\tau)$ is Dedekind's eta-function~\cite{Zagier}. It is known that
$u_7^4$ is a Hauptmodul for $X_1(7)$. As a function of $t:=1728/j$,
where $j$ is the standard elliptic $j$-invariant, $u_7$ satisfies the
algebraic equation
$$
t(49u_7^8+13u_7^4+1)(7^4u_7^8+245u_7^4+1)^3-1728u_7^4=0.
$$
It also satisfies a fourth order differential equation of type G-II in
terms of $x=t^{-1}$
\begin{align}
\label{hurwitz-diff-eqn}
\begin{split}
x^2(x-1)^2\frac{d^4u_7}{dx^4}&+x(x - 1)(7x - 4)\frac{d^3u_7}{dx^3}
+(573/56x^2 - 5899/504x + 20/9)\frac{d^2u_7}{dx^2} \\
&+(12297/5488x - 39779/24696)\frac{du_7}{dx}-57/87808u_7=0\\
\end{split}
\end{align}
with exponents
$$
\begin{array}{c|l}
t&{\rm exponents}\\
\hline
0:& 0, 1/3, 2/3, 1 \\
1:&0, 1/2, 1, 3/2\\
\infty:&-1/28, 3/28, 1/4, 19/28\\
\end{array}.
$$
We see that this example is a conjugate of that in row~$\# 8$ of
Table~\ref{special-1/3}. 

Consider the following modular functions for $\Gamma(7)$.
    \begin{align*}
    \mathbf{x}(\tau) &\= q^{23/84}\prod_{n\equiv \pm4\Mod{7}}(1-q^n)^{-1},\\
    \mathbf{y}(\tau) &\= q^{11/84}\prod_{n\equiv \pm2\Mod{7}}(1-q^n)^{-1},\\
    \mathbf{z}(\tau) &\= q^{-13/84}\prod_{n\equiv \pm1\Mod{7}}(1-q^n)^{-1}.
    \end{align*}
    A basis of solutions to~\eqref{hurwitz-diff-eqn} is then
    $\mathbf{x}\mathbf{y}\mathbf{z},\mathbf{x}^2\mathbf{y},\mathbf{y}^2\mathbf{z},\mathbf{z}^2\mathbf{x}$. Note that $u_7=\mathbf{x}\mathbf{y}\mathbf{z}$.  This is basically
    given in the original paper of Hurwitz~\cite{Hurwitz}. For more
    details on the associated Klein curve see~\cite{Elkies}.
\section{A family of genus two curves}

Consider the G-II rigid local system $\calG$ with parameters 
$$
(\alpha_1,\alpha_2) = (1/6,5/6),
\quad
(\beta_1,\beta_2)= (1/6,5/6),
\quad
(\gamma_1,\gamma_2,\gamma_3,\gamma_4)= (1/5,2/5,3/5,4/5)
$$
and trace field~$\QQ$. We see from row~$\# 35$ of
Table~\ref{ratnl-indef} that the obstruction vanishes and hence it is
definable over~$\QQ$.  We find the following concrete realization
$$
T_0=\left(\begin{array}{rrrr}
 1&0&1&0\\
 1&1&1&1\\
 -1&0&0&-1\\
 0&0&0&1\\
 \end{array}
\right),
\quad
T_1=
\left(\begin{array}{rrrr}
 1&0&1&0\\
 0&1&0&1\\
 -1&0&0&0\\
 0&-1&0&0\\
 \end{array}
\right),
\quad
T_\infty=
\left(
\begin{array}{rrrr}
 -1&0&-1&-1\\
 0&0&0&-1\\
 1&0&0&0\\
 -1&1&0&0\\
\end{array}
\right)
$$
Computing the invariant Hermitian form we find that these matrices are
symplectic. Let
$\Gamma:=\langle T_0,T_1,T_\infty\rangle \subseteq\Sp_4(\ZZ)$ be the
geometric monodromy group.

We will show that $\calG$ arises from $H^1$ of a family of genus two
curves (so it is motivic). To find these we use an argument we learned
from D.~Roberts. We will see that the $\Gamma$ equals the geometric
monodromy of a finite monodromy G-II modulo $2$ (denoted $\calG_1$
below) and use it to produce a family of polynomials of degree $6$
which give rise to the desired curves.

 Bender~\cite{Bender} has given the following generators
for the symplectic group $\Sp_4(\ZZ)$.
$$
K:=\left(\begin{array}{rrrr}
 1&0&0&0\\
 1&-1&0&0\\
 0&0&1&1\\
 0&0&0&-1\\
 \end{array}
\right),
\qquad
L:=
\left(
\begin{array}{rrrr}
 0&0&-1&0\\
 0&0&0&-1\\
 1&0&1&0\\
 0&1&0&0\\
 \end{array}
\right)$$
In terms of these generators we have
$$
T_1=(KL^{-2})^3L^{-1}, \qquad
T_\infty=(KL^{-5})^4.
$$
We can easily verify using MAGMA that the geometric monodromy group
$\Gamma:=\langle T_0,T_1,T_\infty\rangle\leq \Sp_4(\ZZ)$ is the unique
subgroup of index two, namely, the commutator subgroup of $\Sp_4(\ZZ)$.
Here are the calculations
\begin{verbatim}
> G<K,L>:=Group<K,L | K^2=1, L^12=1, K*L^7*K*L^5*K*L = L*K*L^5*K*L^7*K,
L^2*K*L^4*K*L^5*K*L^7*K = K*L^5*K*L^7*K*L^2*K*L^4,
L^3*K*L^3*K*L^5*K*L^7*K = K*L^5*K*L^7*K*L^3*K*L^3,
(L^2*K*L^5*K*L^7*K)^2 = (K*L^5*K*L^7*K*L^2)^2,
L*(L^6*K*L^5*K*L^7*K)^2 = (L^6*K*L^5*K*L^7*K)^2*L,
(K*L^5)^5 = (L^6*K*L^5*K*L^7*K)^2>;
> H<T1,Tinf> := sub<G | (K*L^(-2))^3*L^(-1), (K*L^(-5))^4>;
> Index(G,H);
2
> H eq DerivedSubgroup(G);
true
\end{verbatim}
The quotient of $\Sp_4(\ZZ)$ by its level two congruence subgroup is
isomorphic to $\Sp_4(\FF_2)$, which is known to be isomorphic to
$S_6$. We see that $\Gamma$  maps
surjectively to $A_6$ under the projection map $f$.
\begin{verbatim}
> U:=SymmetricGroup(6);
> homs := Homomorphisms(G, U : Limit := 1);
> homs;
[
    Homomorphism of GrpFP: G into GrpPerm: U, Degree 6, Order 2^4 * 3^2 * 5 
    induced by
        a |--> (1, 2)(3, 4)(5, 6)
        b |--> (1, 2, 3)(4, 5)
]
> f:=homs[1];
> f(T1);
(1, 3, 2)(4, 5, 6)
> f(Tinf);
(1, 4, 6, 5, 3)
> f(Tinf^(-1)*T1^(-1));
(2, 3, 4)
\end{verbatim}

Consider now the G-II rigid local system $\calG_1$ with parameters
$$
(\alpha_1,\alpha_2) = (1/3,2/3),
\quad
(\beta_1,\beta_2)= (1/3,2/3),
\quad
(\gamma_1,\gamma_2,\gamma_3,\gamma_4)= (1/5,2/5,3/5,4/5).
$$
Its trace field is $\QQ$ and indeed we find it in row~$\# 5$ of
Table~\ref{ratnl-defn} among those with finite monodromy.  Since
$\mu=-1$ the system is not definable over $\QQ$.  Using a realization
over $\QQ(\sqrt{-3})$ with MAGMA we find that the geometric monodromy
group is isomorphic to $\SL_2(\FF_9)$, a central extension of $A_6$ by
$C_2$. This group has two irreducible representations of degree four
with Schur index two.

Here is the calculation with MAGMA.
\begin{verbatim}
function goursat(q)
         K<a>:=GF(q);
         R<x>:=PolynomialRing(K);
         w:=RootsInSplittingField(x^2+x+1)[1][1];
         T0:=[w + 1, 0, -1, 0, w, 1, -1, -1, w, 0, 0, (-w + 1)/w,
              0, 0, 0, 1/w];
         T1:= [-1, 0, 0, 1/w, -w - 1, -1, 1, 1/w, -w, -1, 0, (w +1)/w,
               -w, 0, 0, 0]; 
         Tinf:=[w, 1/w, -1/w, (-w^2 - 1)/w^2, w^2 + 2*w + 1, 1/w,
                (-w - 1)/w, (-w^3 -    w^2 - 1)/w^2, w^2 + w, (w +1)/w,
                (-w - 1)/w, (w^3 + w^2 +1)/-w^2, w^2, 0, 0, -w]; 
         G:=MatrixGroup<4,K|T0,T1,Tinf>;
         return G;
end function;
> G:=goursat(101^2);
> #G;
720
> z:=IsIsomorphic(G,SL(2,9));
> z;
true
> Z:=Center(G);
> #Z;
2
> Z:=Center(G);
> G/Z;
Permutation group acting on a set of cardinality 6
Order = 360 = 2^3 * 3^2 * 5
    (1, 2, 4)
    (1, 3, 2)(4, 5, 6)
    (2, 3, 4, 6, 5)
\end{verbatim}

Note that the parameters for $\calG$ and $\calG_1$ are equal up to
fractions with denominator $2$. This means that their respective local
monodromies are the same modulo $2$. It is clear, for example, that
the geometric monodromy of $\calG_1$ is isomorphic to
$A_6\cong \PSL_2(\FF_9)$ modulo $2$ as the center acts by $\pm 1$.

The three even permutations
$$
\sigma_0:=f(T_0)=(2, 3, 4), \sigma_1:=f(T_1)=(1, 3, 2)(4, 5, 6),
\sigma_\infty:=f(T_\infty)=(1, 4, 6, 5, 3)
$$
we computed above generate $A_6$ and correspond to a Belyi map with
cycle type $31^3,3^2,51$.  D. Roberts showed us
how this map is given by the following polynomial
$$
P(x,t):=x^3(x^3+3x^2-5)-t(3x-1)
$$
 Indeed we have
\begin{align*}
P(x,0)&=x^3(x^3+3x^2-5)\\
P(x,1)&=(x^2 + x - 1)^3\\
P(x,\infty)&=3x-1
\end{align*}
We consider the family of genus two curves defined by the
hyperelliptic equation
$$
C_t:\quad y^2=4P(x,t).
$$
Its Igusa invariants are
\begin{align*}
J_2&=2^2\cdot3\cdot5^2\cdot(4t +1)\\
J_4&=2\cdot3\cdot5^4\cdot(4t +1)^2\\
J_6&=2^2\cdot5^4\cdot(736t^3 + 2928t^2 - 564t + 25)\\
J_8&=3\cdot5^6\cdot(4t +1)(1856t^3 - 8112t^2 + 3156t - 25)\\
J_{10}&=2^8\cdot3^6\cdot5^5\cdot(t - 1)^4t^2
\end{align*}

By construction the Galois representation on the two torsion of the
Jacobian of $C_t$ for a generic $t\in \QQ$ is congruent modulo two to
that of the Artin representation associated to the Belyi map. We
therefore expect that the motive $H^1(C_t,\QQ)$ corresponds to
$\calG$. 

We check that this indeed the case by computing the linear
differential equation satisfied by periods of $C_t$.  Starting with
$\omega:=dx/y$ we apply $ D:=d/dt.$ reducing at each stage to a
representative differential form of the type $p(x)/ydx$ with $p$ of
degree at most five modulo exact differentials.  We then look for a
linear relation among $\omega,D\omega \ldots, D^4\omega$.

In this way we find that $\omega$ is annihilated modulo exact
differentials by the differential operator
\begin{align*}
&t^2(t-1)^4D^4 + t(t-1)^3(9t-5)D^3 
+ \frac1{180}(t-1)^2(3456t^2-3281t+715)D^2\\
&+ \frac1{450}(t-1)^2(3888t-2473)D +\frac1{810000}(186624t^2 - 378373t +
 169874).
\end{align*}
It is easy to verify that the differential equation is of
the expected type G-II with exponents
$$
(-1/6, 0, 1/6, 1),\quad( -1/6, 1/6, 5/6, 7/6), \quad (2/5, 3/5, 4/5, 6/5)
$$
at $t=0,1,\infty$ respectively.

\section{Infinite families}
\label{inf-families}
Considering the stringent conditions required for the invariant
Hermitian form $H$ to be definite, it can seem unlikely that there
would be infinitely many examples where $H$ and all its Galois
conjugates are definite.  However, just as for
hypergeometrics~\cite[Theorem~5.8]{Beukers-Heckman}, this is indeed
the case. Moreover, again like hypergeometrics, they come in families
all of which have the (finite) geometric monodromy group
$\Gamma\subseteq \SL_4(\CC)$ acting imprimitively (see
Table~\ref{inf-families-table}). This is not surprising in light of
Jordan's theorem (see the discussion at the end
of~\S\ref{descr-tables}).

In this table $r=m/n$ is an arbitrary rational number in lowest
terms and
$$
\Delta_{1,n}\equiv 
\begin{cases}
C_n^4 & v_2(n)=0\\
C_{n/2}^3C_n & v_2(n)\geq 1\\
\end{cases},
\qquad
\Delta_{2,n}\equiv 
\begin{cases}
C_n^3 & v_2(n)=0\\
C_{n/2}^3C_2 & v_2(n)=1\\
C_{n/4}^3C_4 & v_2(n)\geq 2\\
\end{cases},
$$
where $v_2$ is the valuation at $2$.
\begin{table}[H]
\[
\begin{array}{|c|c|c|c|c|c|c|}
\hline
& \beta,\alpha_1,\alpha_2& \gamma &\Gamma/A&A&|Z(\Gamma)|&\text{Impr}\\
\hline
1&1/2,1/2,r &-r/4,1/4-r/4,1/2-r/4,3/4-r/4  &D_4&\Delta_{1,n}&n&*\\
\hline
2&1/2,1/3,2/3 &r,-r/3,1/3-r/3,2/3-r/3 &A_4&\Delta_{2,n}&\gcd(n,4)&*\\
\hline
\end{array}
\]
\caption{Infinite families}
\label{inf-families-table}
\end{table}

We will show that in fact there are such examples for all of the cases
considered by Simpson~\cite{Simpson} with $g=0,k=3$ punctures and one
partition equal to $1^n$ for some $n$. For all of these systems there
are infinite families of examples all lying in a single geodesic in
the positive components $\TT_+^\irr$.

\subsection{Rational powers}
\label{ratnl-powers}
We start by showing that the rational powers of algebraic functions
satisfy differential equations of certain fixed order.
\begin{proposition}
\label{alg-diff-eqn}
    Let $f(t)$ be an algebraic function of degree $m$. Then for all 
    $r\in\QQ$ the function $f^r$ satisfies $\mathcal{L}_rf^r = 0$,
    where $\mathcal{L}_r$ is a differential operator of order $m$, whose 
    coefficients depend polynomially on $r$.
\end{proposition}
\begin{proof}
    Let $P(t,y)=0$ be the defining equation for $f$, 
    and let $y_1(t),\dots,y_m(t)$ be its solutions. Denote
    by $W(f_1,\dots,f_n)$ the Wronskian determinant and let us write
    \[\mathcal{L}_r[y] = \frac{W(y,y_1^r,\dots,y_m^r)}{W(y_1^r,\dots,y_m^r)}
    = y^{(m)}+A_{m-1}y^{(m-1)}+\dots+A_0.\]
    Define polynomial differential operators $D_n$ by
    \[\frac{f_i^{(n)}}{f_i}\es{=}D_n\Big(\frac{f_i'}{f_i}\Big).\]
    Then $D_{0}(f)=1$, $D_1(f)=f$, $D_2(f)=f'+f^2$,
    and in general they are defined by the recursion 
    $D_{n+1}(f) = (D_n(f))'+D_n(f)f$.
    In terms of these operators we can write the Wronskian in 
    terms of logarithmic derivatives as
    $W(f_1,\dots,f_n)=f_1\dots f_n\,\det(D_i(f_j'/f_j))_{i,j}$.
    Expanding the determinants in the definition of $\mathcal{L}_r$
    shows that $A_k$ can be expressed as rational functions in $r$
    whose coefficients are symmetric expressions in $y_1,\dots,y_m$ 
    and their derivatives and thus are rational functions of $t$.
    
    For generic $r$ (more precisely, whenever $W(y_1^r,\dots,y_m^r)\ne 0$)
    the functions $y_1^r,\dots,y_m^r$ form the full space of solutions of $\mathcal{L}_r$, and thus the singularities of $\mathcal{L}_r$ are 
    contained in the set of singularities of $y_1^r,\dots,y_m^r$ 
    together with the set of points $t$ where one of $y_i$ becomes $0$. 
    In terms of the defining equation $P(t,y)=p_m(t)y^m+\dots+p_0(t)$, 
    these are exactly the values of $t$ where $p_0(t)p_m(t)$ or the 
    discriminant of $P$ vanishes. 
\end{proof}

For instance, consider $u_2(\tau):=(\eta(2\tau)/\eta(\tau))^{24}$. It
is known that $u_2$ is a Hauptmodul for $\Gamma_0(2)$ and satisfies
the algebraic equation
$$
\label{t2-eqn}
 \qquad A_2(u,t):=t(1+256u)^3-1728u=0, \qquad t:=1728/j,
$$
where $j$ is the standard elliptic $j$-invariant.
For any $r\in\QQ$ its $r$-th power is annihilated
by the $3$-rd order differential operator
    \begin{equation*}
    x^3(x-1)\frac{d^3}{dx^3}
    +x^2(4x-5/2)\frac{d^2}{dx^2}
    +x\Big(\frac{20}{9}x + \frac{(3r+2)(r-1)}{4}\Big)\frac{d}{dx}
    +\frac{r^2(r-1)}{4},
    \end{equation*}
when expressed in $x=t(\tau)=1728/j(\tau)$. The local exponents are
	$$
	0: (r,-r/2,(1-r)/2), \quad
	1: (0, 1/2, 1), \quad
	\infty: (0, 1/3, 2/3).
	$$
Hence these are hypergeometric equations.

\subsection{Simpson even and odd families}
\label{simpson-even-odd}
Given an odd positive integer $N>1$ consider the hypergeometric series
$$
f_N(t):={}_2F_1\left(
\begin{array}{cc}
\frac{N-1}{2N}&\frac{N+1}{2N}\\
\frac32& 1
\end{array}
|\,t\right).
$$
It is an algebraic function of $t$ satisfying $P_N(f_N(t),t)=0$ for a
polynomial $P_N(u,t)\in \QQ[u,t]$. This polynomial $P_N$ can be given
explicitly; the information we need is the shape of its Newton polygon
$\Delta_N$, convex hull of the $(r,s)\in\ZZ^2$ for which the monomial
$u^rt^s$ in $P_N$ has a non-zero coefficient. The polygon $\Delta_N$
is in fact the triangle of vertices $(0,0),(1,0),(N,(N-1)/2)$. If we
orient the boundary of the triangle counter-clockwise starting at the
origin the three sides have slopes $0,1/2,-\kappa_N$, where
$\kappa_N:=(N-1)/2N$, respectively.

For example, for $N=5$ we find
$$
P_5(t,u)=16u^5t^2 - 500u^3t + 3125u - 3125
$$
and its Newton polygon
\begin{figure}
\begin{center}
\begin{tikzpicture}

\draw[step=1.0,black,thin] (0,0) grid (5,2);

\draw[fill=black] (0,0) circle (.1);
\draw[fill=black] (1,0) circle (.1);
\draw[fill=black] (3,1) circle (.1);
\draw[fill=black] (5,2) circle (.1);

\draw[thick] (0,0) -- (1,0);
\draw[thick] (1,0) -- (5,2);
\draw[thick] (5,2) -- (0,0);

\end{tikzpicture}
\end{center}
\caption{$\Delta_5$} \label{fig:Delta-5}
\end{figure}
with sides of slope $0,1/2,-2/5$.

At a zero or pole of $f_N(t)$ we have $t=0,1$ or $t=\infty$. Hence by
Proposition~\ref{alg-diff-eqn} $f_N(t)^r$ for $r\in \QQ$ satisfies a
linear differential equation of order $N$ with singularities only at
$t=0,1,\infty$. 

In general, the exponents at $t=0,\infty$ of the differential equation
satisfied by an algebraic function $f$ of this kind can be read-off
from its Newton polygon $\Delta$. It can be proved that these are as
follows.

Assume that the Newton polygon of $f$ has no vertical segments. Then
there exist unique leftmost and rightmost vertices of $\Delta$, say
$p,q$ respectively. Let $l$ be the line joining $p$ and $q$. We can
distinguish the top and bottom sides of $\Delta$ as those above and
below $l$ respectively.

For each slope $\kappa\in\QQ$ of a side $\delta$ of the Newton polygon
consider the sequence
$$
[\kappa]:=0-\kappa r,\frac1d-\kappa r,\ldots,\frac{e-1}d-\kappa r,
$$ 
where $d$ is the denominator of $\kappa$ and $e$ is the horizontal
width of $\delta$.

The exponents at $t=0$ for $f^r$ are $[\kappa_1],[\kappa_2],\ldots$,
where $\kappa_1,\kappa_2,\ldots$ runs over the slopes of the bottom
sides.  The exponents at $t=\infty$ are similarly determined by the
slopes of the top sides.  The exponents at $t=1$ are independent of
$r$ and can be computed directly from the Newton polygon of
$f(u,t+1)$.

In the case of $f_N$ the bottom slopes are $0,1/2$, the only top slope
is $-\kappa_N$ and we obtain the following
$$
\begin{array}{c|l}
t& {\rm exponents}\\
\hline
0 &0,-r/2,1/2-r/2,\ldots,(N-2)/2-r/2\\
1 &0,1/2,1,\ldots,(N-1)/2\\
\infty &\kappa_N r,1/N+\kappa_N r,\ldots,(N-1)/N+\kappa_N r\\
\end{array}.
$$
For example, when $N=5$ these exponents are
$$
\begin{array}{c|l}
t& {\rm exponents}\\
\hline
0 &0,-r/2,1/2-r/2,1-r/2,3/2-r/2\\
1 &0,1/2,1,3/2,2\\
\infty &2/5r,1/5+2/5r,3/5+2/5r,3/5+2/5r,4/5+2/5r
\end{array}.
$$

If $r\in \QQ$ is not an integer then the multiplicity of these
exponents is $(m,m,1),(m+1,m),(1,\ldots,1)$, where $m:=(N-1)/2$. These
are precisely the multiplicities of Simpson's odd rank case family of
rigid local systems. Hence we have obtained a geodesic
completely contained in the positive components $\TT^\irr_+$ of this system's
parameter space.

A completely analogous discussion holds for $N$ even with the same
definition of $f_N$. Here it is more convenient to consider the
algebraic equation for $f_N^2$, which is the hypergeometric function
$$
f_N(t)^2={}_2F_1\left(
\begin{array}{cc}
\frac{N-1}N&\frac{N+1}N\\
\frac32& 2
\end{array}
| t\right).
$$
It satisfies an algebraic equation of degree $N$ with Newton polygon
the triangle of vertices $(0,0),(1,0),(N,N)$. The exponents are the
same as in the case $N$ odd. The multiplicities however are now
$(1,m-1,m),(m,m),(1,\ldots,1)$, where $m:=N/2$. These are the
multiplicities of Simpson's even rank case family of rigid local
systems~\cite{Simpson}. Again we have obtained a geodesic completely
contained in the positive components $\TT^\irr_+$ of the parameter
space.  For example, for $N=4$ we get a geodesic for Goursat G-II up
to a twist.
$$
\begin{array}{c|l}
t& {\rm exponents}\\
\hline
0 &0,-r,1/2-r,1-r\\
1 &0,1/2,1,3/2\\
\infty &3/4r,1/4+3/4r,1/2+3/4r,3/4+3/4r
\end{array}
$$

\subsection{Simpson extra case of rank $6$}
The Hauptmodul $u_5:=(\eta(5\tau)/\eta(\tau))^6$ satisfies the
equation
$$
A_5(u,t)=t(3125u^2+250u+1)^3-1728u=0,
$$
whose Newton polygon is a triangle with vertices $(0,1),(4,1),(0,1)$
and slopes $-1,1/5,0$.

Fractional powers of $u_5(\tau) = (\eta(5\tau)/\eta(\tau))^6$ give the
rank $6$ extra case rigid local system of Simpson.  Explicitly, in
terms of $x=1/t(\tau)$, $u_5^r$ satisfies the differential equation
    \begin{align*}
    &x^4(x-1)^2\frac{d^6}{dx^6}
    +x^3(x-1)(17x-10)\frac{d^5}{dx^5}\\
    +&x^2\Big(\frac{-3r^2+2r+432}{5}x^2 + \frac{108r^2-72r-18377}{180}x + \frac{220}{9}\Big)\frac{d^4}{dx^4}\\
    +&x\Big(\frac{8r^3 - 118r^2 + 74r + 3720}{25}x^2 - \frac{576r^3 - 5796r^2 + 3528r + 209915}{1800}x + \frac{40}{3}\Big)\frac{d^3}{dx^3}\\
    +&\Big(\frac{-45r^4 + 860r^3 - 5145r^2 + 2950r + 45024}{625}x^2 - \frac{576r^3 - 2676r^2 + 1448r + 26135}{900} x + \frac{40}{81}\Big)\frac{d^2}{dx^2}\\
    +&\Big(\frac{24r^5 - 415r^4 + 2790r^3 - 7925r^2 + 3846r + 15120}{3125}x - \frac{64r^3 - 164r^2 + 72r + 315}{900} \Big)\frac{d}{dx}\\
    -&\frac{r^2(r-1)(r-2)(r-3)(r-4)}{3125}.
	\end{align*}
        The exponents of this equation for generic $r \in \QQ$ are
$$
\begin{array}{c|l}
t& {\rm exponents}\\
\hline
0 &r,-r/5,1/5-r/5,2/5-r/5,3/5-r/5,4/5-r/5\\
1 &0,1/2,1,3/2,2,3\\
\infty &0,1/3,2/3,1,4/3,5/3\\
\end{array}
$$
with multiplicities $(4, 2), (2, 2, 2), (1,\ldots,1)$. This is then a
geodesic in the positive components $\TT^\irr_+$ of Simpson's extra
case. A special case reduces to a hypergeometric series
$$
\frac{u_5}t=  {}_3F_2\left(
\begin{array}{ccc}
\tfrac12&\tfrac16&\tfrac56\\
\tfrac45&\tfrac65&1
\end{array}
| t\right)^6.
$$

For $r=1,3$ the equation reduces to a hypergeometric equation of order
$3$, and for $r=2,4,8,14$ it reduces to an equation of order $5$ with
rigid monodromy of Simpson's odd type.

\subsection{Hypergeometric}
Theorem~5.8 in~\cite{Beukers-Heckman} describes a geodesic in the case
of hypergeometric rigid local systems. This can be made explicit in
terms of fractional powers of a fixed algebraic function like all the
previous examples. We have already encountered one case (see the
example at the end of~\S\ref{ratnl-powers}). We illustrate this
further with an instance of rank $5$.

Consider the algebraic equation
$$
u(1-u)^4-\frac{4^4}{5^5}t=0.
$$
We can solve for $u$ as a function of $t$ by inversion. Let 
$$
f(t)=\frac{5^5}{4^4}\frac{u(t)}t= 1 + t_1 + \frac{13}8t_1^2 + \frac{51}{16}t_1^3 +
\frac{1771}{256}t_1^4 + \frac{4095}{256}t_1^5 +O(t_1^6), \qquad
t_1:=\frac{5^5}{4^4}t. 
$$
We find that
$$
f(t)= {}_4F_3\left(
\begin{array}{cccc}
\tfrac45&\tfrac65&\tfrac75&\tfrac85\\
\tfrac32&\tfrac54&\tfrac74&2
\end{array}
| t\right).
$$
We compute the local exponents of $f^r$ for $r\in \QQ$ and find
$$
\begin{array}{c|l}
t& {\rm exponents}\\
\hline
0 &0,-r,1/4-r,1/2-r,3/4-r\\
1 &1/2,0,1,2,3\\
\infty & 4/5r,1/5+4/5r,2/5+4/5r,3/5+4/5r,4/5+4/5r\\
\end{array}
$$
We obtain the following identity between hypergeometric functions
$$
{}_4F_3\left(
\begin{array}{cccc}
\tfrac45&\tfrac65&\tfrac75&\tfrac85\\
\tfrac32&\tfrac54&\tfrac74&2
\end{array}
| t\right)^r=
{}_5F_4\left(
\begin{array}{ccccc}
4/5r&1/5+4/5r&2/5+4/5r&3/5+4/5r&4/5+4/5r\\
1+r&3/4+r&1/2+r&1/4+r&1
\end{array}
| t\right).
$$

\subsection{Goursat II}
\label{G-II-inf}

There is an another geodesic in the case G-II apart from that
in~\S\ref{simpson-even-odd} for $n=4$.  Consider the modular unit
$u_3(\tau):=\eta(3\tau)/\eta(\tau)$. It is a classical fact that
$u_3^{12}$ is a Hauptmodul for the modular curve $X_0(3)$ and
satisfies the algebraic equation $A_3(3^6u_3^{12},1728j^{-1})=0$,
where
$$
\label{t3-eqn}
 \qquad A_3(u,t):=t(u+27)(u+3)^3-1728u. 
$$
The Newton polygon of $A_3$ is a triangle with vertices
$(0,1),(4,1),(0,1)$ and slopes $-1,1/3,0$. 
The $4$-th order differential equation satisfied by  $u_3^r$ is
    \begin{multline}
      x^2(x-1)^2\frac{d^4}{dx^4} +x(x-1)(7x-4)\frac{d^3}{dx^3}
      +\Big(\frac{-6r^2+3r+92}{9}x^2 + \frac{24r^2 - 12r - 421}{36}x + \frac{20}{9}\Big)\frac{d^2}{dx^2}\\
      +\Big(\frac{8r^3-33r^2+13r+60}{27}x - \frac{64r^3 - 192r^2 + 68r
        + 345}{216}\Big)\frac{d}{dx} -\frac{r^2(r-1)(r-2)}{27},
    \end{multline}
when expressed in $x=1/t(\tau)=j(\tau)/1728$. The exponents of this equation
for generic $r \in \QQ$ are
$$
\begin{array}{c|l}
t& {\rm exponents}\\
\hline
0 &0,1/3,2/3,1\\
1 &0,1/2,1,3/2\\
\infty &r,-r/3,1/3-r/3,2/3-r/3\\
\end{array}
$$

We have
$$
\frac{u_3}t= {}_2F_1\left(
\begin{array}{cc}
\tfrac7 {12}&\tfrac1 4\\
\tfrac4 3&1
\end{array}
| t\right)^4, \qquad t:=\frac{1728} j.
$$

Here is a detailed description of this geodesic. The four
exponents at $\infty$ are
\[\gamma_1=3r,\;\;\gamma_2=-r,\;\;\gamma_3=-r+1/3,\;\;\gamma_4=-r+2/3.\]
We plot these modulo $\ZZ$ as a function of $r$ (see
Figure~\ref{G-II-geodesic-gamma}).
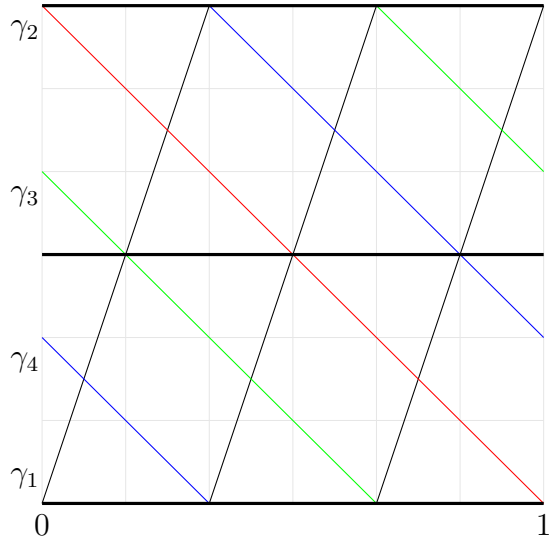
\begin{figure}
\begin{center}
    \begin{tikzpicture}
    \tikzstyle{axis}=[very thick,->]; 
    
    \begin{scope}[scale=1.1]   
    \definecolor{cgray}{rgb}{0.9,0.9,0.9};
    \definecolor{cg2}{rgb}{1,0,0};
    \definecolor{cg3}{rgb}{0,1,0};
    \definecolor{cg4}{rgb}{0,0,1};
    \definecolor{cg1}{rgb}{0,0,0};
    
    \draw[step=1cm,cgray,very thin] (0,-3) grid (6,3);
    \draw[cg1] (0,-3) -- (2,3);
    \draw[cg1] (2,-3) -- (4,3);
    \draw[cg1] (4,-3) -- (6,3);    
    \draw[cg2] (0,3) -- (6,-3);
    \draw[cg3] (0,1) -- (4,-3);
    \draw[cg3] (4,3) -- (6,1);
    \draw[cg4] (0,-1) -- (2,-3);    
    \draw[cg4] (2,3) -- (6,-1);

    \draw[black, very thick] (0,3) -- (6,3);    
    \draw[black, very thick] (0,0) -- (6,0);
    \draw[black, very thick] (0,-3) -- (6,-3);

    \node at (0.1,-2.7)[left] {$\gamma_1$};
    \node at (0.1,2.7)[left] {$\gamma_2$};
    \node at (0.1,0.7)[left] {$\gamma_3$};
    \node at (0.1,-1.3)[left] {$\gamma_4$};

    \node at (0,-3)[below] {$0$};
    \node at (6,-3)[below] {$1$};
    
    \end{scope}
    \end{tikzpicture}
\end{center}
\caption{G-II geodesic $\gamma$}
\label{G-II-geodesic-gamma}
\end{figure}
The first condition ($n_1=2$ in the notation
of~\eqref{goursat-pos-defn-1}) for these parameters for generic $r$ to
be in the positive components $\TT^\irr_+$ is that there are two in
each of the indicated horizontal strips. This is visible in the plot.

The second condition for positive definiteness involves the parameters
    \[\delta_1=2r,\;\;\delta_2=2r+2/3,\;\;\delta_3=2r+1/3,\;\;
      \delta_4=-2r+2/3,\;\;\delta_5=-2r+1/3,\;\;\delta_6=-2r.\] For
    generic $r$ there should be four $\delta$'s in the interval
    $(1/6,5/6)$.  A plot of these as functions of $r$ modulo $\ZZ$ is
    given in Figure~\ref{G-II-geodesic-delta} where this condition is
    visible.

\begin{figure}
\begin{center}
    \begin{tikzpicture}
    \begin{scope}[scale=1.1]
    \tikzstyle{axis}=[very thick,->];
    \definecolor{cgray}{rgb}{0.9,0.9,0.9};
    \definecolor{cr1}{rgb}{0.9,0.9,0.1};
    \definecolor{cr2}{rgb}{0.1,0.9,0.9};
    \definecolor{cr3}{rgb}{0.9,0.1,0.9};
    \definecolor{cr4}{rgb}{0.1,0.1,0.9};
    \definecolor{cr5}{rgb}{0.1,0.9,0.1};
    \definecolor{cr6}{rgb}{0.9,0.1,0.1};
 
    \draw[black, very thick] (0,-2) -- (6,-2);    
    \draw[black, very thick] (0,2) -- (6,2);    
   
    \draw[step=1cm,cgray,very thin] (0,-3) grid (6,3);
    \draw[cr1] (0,-3) -- (3,3);
    \draw[cr1] (3,-3) -- (6,3);
    \draw[cr2]  (0,1) -- (1,3);
    \draw[cr2] (1,-3) -- (4,3);
    \draw[cr2] (4,-3) -- (6,1);
    \draw[cr3] (0,-1) -- (2,3);
    \draw[cr3] (2,-3) -- (5,3);
    \draw[cr3] (5,-3) -- (6,-1);
    \draw[cr6] (0,3) -- (3,-3);
    \draw[cr6] (3,3) -- (6,-3);
    \draw[cr4]  (0,1) -- (2,-3);
    \draw[cr4] (2,3) -- (5,-3);
    \draw[cr4] (5,3) -- (6,1);
    \draw[cr5] (0,-1) -- (1,-3);
    \draw[cr5] (1,3) -- (4,-3);
    \draw[cr5] (4,3) -- (6,-1);
    
    \node at (0,-3)[below] {$0$};
    \node at (6,-3)[below] {$1$};
    
    \node at (0.3,-2.5)[left] {$\delta_1$};
    \node at (0.3,1.5)[left] {$\delta_2$};
    \node at (0.3,-0.5)[left] {$\delta_3$};
    \node at (0.3,0.5)[left] {$\delta_4$};
    \node at (0.3,-1.5)[left] {$\delta_5$};
    \node at (0.3,2.5)[left] {$\delta_6$};
    \end{scope}
    \end{tikzpicture}
\end{center}
\caption{G-II geodesic $\delta$}
\label{G-II-geodesic-delta}
\end{figure}
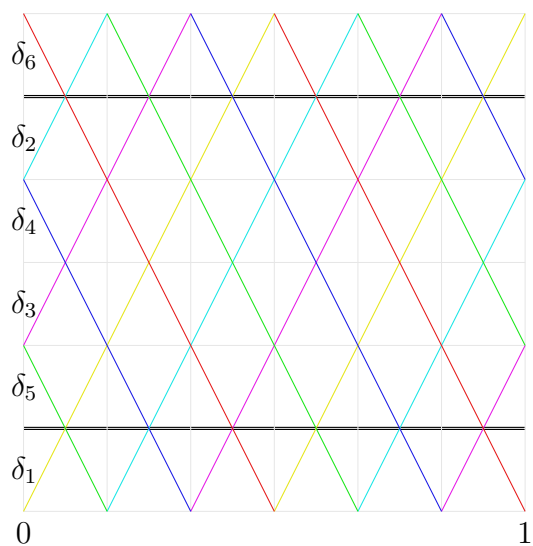

\end{document}